%% file: wave-cov.tex
\documentclass[11pt, a4paper]{amsart}
\usepackage{a4wide}
\usepackage{graphicx,enumitem}
\usepackage{amsmath,amssymb}
\usepackage[english]{babel}
\usepackage{subfigure}
\usepackage{units}
\usepackage{diagbox}
\usepackage{xcolor}
\usepackage{caption}
\usepackage{ulem}
\usepackage{cancel}

\usepackage{amssymb}
\DeclareSymbolFontAlphabet{\amsmathbb}{AMSb}%
\usepackage{mathbbol}
\usepackage{latexsym}

\usepackage[pdftex,
pdftitle={Weak convergence of fully discrete finite element approximations of semilinear hyperbolic SPDE with additive noise},
bookmarksopen,
colorlinks,
linkcolor=black,
urlcolor=black,
citecolor=black
]{hyperref}
\hypersetup{pdfauthor={M. Kov\'acs, A. Lang, and A. Petersson}}

\usepackage{dsfont}

\usepackage{mathtools}

\DeclarePairedDelimiter{\floor}{\lfloor}{\rfloor}

\newcommand{\R}{\amsmathbb{R}}
\newcommand{\D}{\amsmathbb{D}}
\newcommand{\C}{\amsmathbb{C}}
\newcommand{\N}{\amsmathbb{N}}

\newcommand{\cA}{\mathcal{A}}

\newcommand{\cC}{\mathcal{C}}
\newcommand{\cD}{\mathcal{D}} 

\newcommand{\cF}{\mathcal{F}}
\newcommand{\cG}{\mathcal{G}}
\newcommand{\cH}{\mathcal{H}}

\newcommand{\cL}{\mathcal{L}}

\newcommand{\cS}{\mathcal{S}}

\newcommand{\cV}{\mathcal{V}}

\newcommand{\Op}{\operatorname{O}}

\DeclareMathOperator{\E}{\amsmathbb{E}} 

\DeclareMathOperator{\trace}{Tr}

\newcommand{\dd}{\,\mathrm{d}}

\newcommand{\inpro}[3][{}]{ \left\langle #2 , #3 \right\rangle_{#1} }
\newcommand{\norm}[2]{\| #1 \|_{#2}}
\newcommand{\lrnorm}[2]{\left\| #1 \right\|_{#2}}
\newcommand{\Bignorm}[2]{\Big\| #1 \Big\|_{#2}}

\newcommand{\dfloor}[1]{\floor{#1}_{\Delta t}}

\newtheorem{lemma}{Lemma}[section]
\newtheorem{proposition}[lemma]{Proposition}

\newtheorem{theorem}[lemma]{Theorem}
\theoremstyle{remark}

\theoremstyle{definition}

\newtheorem{assumption}[lemma]{Assumption}

\begin{document}
	\title[Weak convergence for semilinear hyperbolic SPDE with additive noise]{Weak convergence of fully discrete finite element approximations of semilinear hyperbolic SPDE with additive noise
	}
	
	\author[M.~Kov\'acs]{Mih\'aly Kov\'acs} \address[Mih\'aly Kov\'acs]{\newline Faculty of Information Technology and Bionics
		\newline P\'azm\'any P\'eter Catholic University
		\newline H-1444 Budapest, P.O. Box 278}. \email[]{kovacs.mihaly@itk.ppke.hu}
	
	\author[A.~Lang]{Annika Lang} \address[Annika Lang]{\newline Department of Mathematical Sciences
		\newline Chalmers University of Technology \& University of Gothenburg
		\newline S--412 96 G\"oteborg, Sweden.} \email[]{annika.lang@chalmers.se}
	
	\author[A.~Petersson]{Andreas Petersson} \address[Andreas Petersson]{\newline Department of Mathematical Sciences
		\newline Chalmers University of Technology \& University of Gothenburg
		\newline S--412 96 G\"oteborg, Sweden.} 
	\email[]{andreas.petersson@chalmers.se}
	
	\thanks{
		Acknowledgement. MK thanks the support of the Swedish Research Council (VR) through Grant~No.~2017-04274 and the Marsden Fund of the Royal Society of New Zealand through Grant~No.~18-UOO-143. AL and AP thank the support of the Knut and Alice Wallenberg foundation, the Swedish Research Council under Reg.~No.~621-2014-3995 and the Wallenberg AI, Autonomous Systems and Software Program (WASP) funded by the Knut and Alice Wallenberg Foundation.
	}
	
	
	\begin{abstract}
		The numerical approximation of the mild solution to a semilinear stochastic wave equation driven by additive noise is considered. A standard finite element method is employed for the spatial approximation and a a rational approximation of the exponential function for the temporal approximation. First, strong convergence of this approximation in both positive and negative order norms is proven. With the help of Malliavin calculus techniques this result is then used to deduce weak convergence rates for the class of twice continuously differentiable test functions with polynomially bounded derivatives. Under appropriate assumptions on the parameters of the equation, the weak rate is found to be essentially twice the strong rate. This extends earlier work by one of the authors to the semilinear setting. Numerical simulations illustrate the theoretical results.
	\end{abstract}
	
	\maketitle
	
	\section{Introduction}
	
	The stochastic wave equation is an evolutionary equation that can be used to model various time dependent phenomena influenced by random forces. One example (see~\cite{D09}) is the vertical displacement $u: [0,T] \times \cD \to \R$ of a DNA string suspended in a liquid, 
		\begin{equation}
		\label{eq:DNA_intro_equation}
		\dd \dot{u} (t) - \Delta u(t) \dd t =  - Qu(t) \dd t + \dd W(t)
		\end{equation}
		for $t \in (0,T]$, $T<\infty$, where $\Delta$ is the Laplacian with suitable boundary conditions on a convex domain $\cD \subset \R^d$, $d=1,2,3$. The first term on the right hand side of~\eqref{eq:DNA_intro_equation} models friction due to viscosity of the fluid, while the Gaussian noise term $\dd W(t)$ corresponds to random bombardment of the DNA string by the fluid's molecules. This noise is white in time with spatial correlation described by the linear operator $Q$ on $L^2(\cD) = L^2(\cD,\R)$, the same operator as in the friction term.
		Thus~\eqref{eq:DNA_intro_equation} can be treated as a stochastic partial differential equation in the It\^o  sense, driven by a Wiener process~$W$ in $L^2(\cD)$.
		
		In this paper, we are concerned with the more general setting that the friction due to viscosity may depend non-linearly on the displacement of the DNA string and that the intensity of the molecular bombardment may vary in time. We thus consider the equation
		\begin{equation}
		\label{eq:intro_stochastic_wave_equation}
		\dd \dot{u} (t) - \Delta u(t) \dd t =  F(t,u(t)) \dd t + G(t) \dd W(t),
		\end{equation}
		with the goal of analyzing errors stemming from the approximation of this equation by finite elements and a rational approximation of the exponential function. The Laplacian $\Delta$ is assumed to satisfy zero Dirichlet boundary conditions, i.e., $u(t) = 0$ on $\partial \cD$ for all times $t \in (0,T]$, and the equation has initial conditions $u(0) =u_0$ and $\dot{u}(0) =v_0$. 
	
	In general, \eqref{eq:intro_stochastic_wave_equation} cannot be solved analytically. The question of how to find an approximation $\hat{u}$ of $u$ and how to evaluate the quality of such an approximation a priori is therefore of great importance if one wants to use this equation in practice. 
	In order to implement an approximation on a computer, the equation is typically discretized both in the spatial and temporal domain, in which case the resulting approximation $\hat{u}$ is said to be fully discrete. In the literature, the quality of $\hat u$ is in general evaluated by analyzing the rate of decay of the strong error $\E[\norm{u - \hat{u}}{L^2(\cD)}^2]^{1/2}$ (see~\cite{ACLW16,CY07, CLS13, CQ15, CJL19, JJW15, KLL12, KLL13, KLS15, QS06, W06, W15,WGT14}). Comparatively few results (see~\cite{CJL19,H10,JJW15, KLL12, KLL13, KLS15,W15}) exist on the rate for the weak error $|\E[\phi(u) - \phi(\hat{u})]|$, where $\phi\colon L^2(\cD) \to \R$ is a sufficiently smooth real-valued test function. Of the results cited, only~\cite{H10} provides a weak convergence result for a fully discrete approximation of a semilinear stochastic wave equation. If $\phi$ is (locally) Lipschitz, the weak error can be bounded by the strong error, but in analysis the rate of decay of the weak error as one considers finer and finer approximations is often found to be twice the rate of the strong error. 
	
	The outline of our paper is the following. In Section~\ref{sec:stoch_wave}, we analyze~\eqref{eq:intro_stochastic_wave_equation} in a more general, abstract, Hilbert space setting and show spatial and temporal regularity results under mild assumptions on $F$. 
	
	In Section~\ref{sec:appr_and_conv} we deduce strong and weak error rates for the approximation of the so called mild solution $u$ of~\eqref{eq:intro_stochastic_wave_equation} by means of a finite element approximation (by piecewise linear or quadratic functions) in space and a rational approximation of the exponential function in time, generalizing the result of~\cite{KLL13} to the semilinear setting. This approach sets the paper apart from several recent works (e.g., \cite{ACLW16, CLS13, CQ15, CJL19, W15, WGT14}) on the stochastic wave equation that consider trigonometric integrators for the temporal approximation. There are situations when such integrators could be better suited such as highly oscillatory data but for complicated domain geometries the algorithms in the present article could be more advantageous from an implementation point of view, since they do not require any knowledge of the eigenfunctions of~$\Delta$ or its discrete counterpart.
	
	For the analysis we take a similar approach as the author of~\cite{W15}, by using negative norm strong convergence rates in our analysis of the weak error. However, instead of using Kolmogorov's equation and the It\^o formula, we complete the analysis by means of Malliavin calculus. Our results are applicable under slightly more general assumptions on $F(t,\cdot)$ compared to~\cite{W15}, specifically when $F(t,\cdot)$ is a Nemytskij operator, i.e., when $F(t,u)(x) = f(t,u(x))$ for $u \in L^2(\cD)$ and almost every $x \in \cD$. Here $f(t,\cdot)$ is a real-valued function of at most linear growth, with bounded and Lipschitz-continuous first derivative.
	The test function $\phi$ is assumed to be twice G\^ateaux differentiable with polynomially bounded derivatives.
	
	Section~\ref{sec:examples_simulation} finishes the main part of the paper with examples in which it is noted that when $F$ is a sufficiently smooth Nemytskij operator, the derived weak convergence rates are essentially twice as big as the strong convergence rates, provided that the initial value is smooth, for $d=1,2$ when the covariance operator $Q$ of $W$ is of trace-class, and $d=1$ when $Q=I$. Numerical simulations in $d=1$ illustrate our theoretical results. 
	
	In Appendix~\ref{sec:appendix}, which completes the paper, it is shown that a sufficiently smooth Nemytskij operator fulfills the assumptions of Section~\ref{sec:appr_and_conv}.
	
	Throughout the paper, we adopt the notion of generic constants, which is to say that the symbol $C$ is used to denote a positive and finite number which may vary from occurrence to occurrence and is independent of any parameter of interest, such as spatial and temporal step sizes in a numerical method. We use the expression $a \lesssim b$ to denote the existence of a generic constant $C$ such that $a \le C b$.

	\section{The stochastic wave equation}
	\label{sec:stoch_wave}
	
	
	In this section the stochastic wave equation is presented along with necessary background material from probability theory and functional analysis. We use the semigroup approach of~\cite{DPZ14} and refer to this monograph for more details on the material covered here. The equation is treated in an abstract Hilbert space setting, while in the next section we restrict ourselves to the setting in which the solution takes values in the Hilbert space $L^2(\cD)$ where $\cD \subset \R^d$, $d=1,2,3$, denotes the underlying domain. 
	
	Let $(H, \inpro[H]{\cdot}{\cdot}, \norm{\cdot}{H})$ and $(U, \inpro[U]{\cdot}{\cdot}, \norm{\cdot}{U})$ be real separable Hilbert spaces. We denote by $(\cL(H,U), \norm{\cdot}{\cL(H,U)})$ the space of bounded linear operators from $H$ to $U$ equipped with the usual operator norm and by $(\cL_1(H,U), \norm{\cdot}{\cL_1(H,U)})$ and $(\cL_2(H,U), \inpro[\cL_2(H,U)]{\cdot}{\cdot}, \norm{\cdot}{\cL_2(H,U)})$ the subsets of trace-class and Hilbert--Schmidt operators, respectively. We use the shorthand notations $\cL(H) = \cL(H,H)$, $\cL_1(H) = \cL_1(H,H)$ and $\cL_2(H) = \cL_2(H,H)$. Note that if $(V, \inpro[V]{\cdot}{\cdot}, \norm{\cdot}{V})$ is another Hilbert space and if $\Gamma_2 \in \cL(U,V)$, $\Gamma_1 \in \cL_i(H,U)$, $i \in \{1,2\}$, then $\Gamma_2 \Gamma_1 \in \cL_i(H,V)$ and 
	\begin{equation}
	\label{eq:schatten_bound_1}
	\norm{\Gamma_2 \Gamma_1}{\cL_i(H,V)} \le \norm{\Gamma_2}{\cL(U,V)} \norm{\Gamma_1}{\cL_i(H,U)}.
	\end{equation}
	Similarly, if $\Gamma_1 \in \cL_i(V,H)$, $i \in \{1,2\}$, then $\Gamma_1 \Gamma_2 \in \cL_i(U,H)$ and 
	\begin{equation}
	\label{eq:schatten_bound_2}
	\norm{\Gamma_1 \Gamma_2}{\cL_i(U,H)} \le \norm{\Gamma_1}{\cL_i(V,H)} \norm{\Gamma_2}{\cL(U,V)}.
	\end{equation}
	The trace of $\Gamma \in \cL_1(H)$ is, for an orthonormal basis $(e_j)_{j=1}^\infty$ of $H$, defined by $\trace(\Gamma) = \sum^\infty_{j=1} \inpro[H]{\Gamma e_j}{e_j}$ and is independent of the choice of basis. If $\Gamma_1 \in \cL_1(H,U)$ and $\Gamma_2 \in \cL(U,H)$ then \begin{equation}
	\label{eq:trace_cyclic}
	\trace(\Gamma_1 \Gamma_2) = \trace(\Gamma_2 \Gamma_1).
	\end{equation} 
	
	We will have reason to use spaces of G\^ateaux differentiable mappings, which we define in the same way as the authors of~\cite{AKL16}. By $\cC(H,U)$ we denote the space of continuous mappings from $H$ to $U$ and by $\cG^1(H,U) \subset \cC(H,U)$ the space of G\^ateaux differentiable mappings with strongly continuous derivatives, i.e., the space of all continuous mappings $\varphi\colon H \to U$ such that
	\begin{equation*}
	\varphi'(u)v = \lim_{\epsilon\to 0} \frac{1}{\epsilon} \left(\varphi(u + \epsilon v) - \varphi(u)\right)
	\end{equation*}
	exists as a limit in $U$ for all $u, v \in H$, that $\varphi'(u) \in \cL(H,U)$ for all $u \in H$ and that the mapping $H \ni u \mapsto \varphi'(u) v$ is continuous for all $v \in H$. If in addition $\varphi' \in \cC(H,\cL(H,U))$, then $\varphi \in \cC^{1}(H,U)$, the space of Fr\'echet differentiable mappings. By $\cG^2(H,U) \subset \cG^1(H,U)$ we denote the space of all mappings $\varphi \in \cG^1(H,U)$ such that
	\begin{equation*}
	\varphi''(u)(v,w) = \lim_{\epsilon\to 0} \frac{1}{\epsilon} \left(\varphi'(u + \epsilon w) v - \varphi'(u)v\right)
	\end{equation*}
	exists as a limit in $U$ for all $u, v, w \in H$, that $\varphi''(u)\colon H \times H \to U \in \cL^{[2]}(H,U)$, the space of all bounded bilinear mappings, for all $u \in H$, that $\varphi''(u)$ is symmetric for all $u \in H$, and that the mapping $H \ni u \mapsto \varphi''(u) (v,w)$ is continuous for all $v,w \in H$. 
	For $n = 1,2$, we denote by $\cG^n_\mathrm{b}(H,U)$ and $\cG^n_\mathrm{p}(H,U)$ the sets of all $\varphi \in \cG^n(H,U)$ such that all derivatives of $\varphi$ (but not necessarily $\varphi$ itself) are bounded and polynomially bounded, respectively, with $\cC^1_\mathrm{b}(H,U)$ and $\cC^1_\mathrm{p}(H,U)$ defined analogously.  We use the shorthand notations $\cG^n(H) = \cG^n(H,H)$, $\cG^n_\mathrm{b}(H) = \cG^n_\mathrm{b}(H,H)$ and $\cG^n_\mathrm{p}(H) = \cG^n_\mathrm{p}(H,H)$, and similarly for the spaces of Fr\'echet differentiable mappings. For $\varphi \in \cG^1_\mathrm{p}(H,U)$ and $u,v \in H$ the mean value theorem holds in $U$, i.e., 
	\begin{equation*}
	\varphi(u) - \varphi(v) = \int^1_0 \varphi'(v + s(u-v)) (u-v) \dd s.
	\end{equation*}
	
	For $0<T<\infty$, let $(\Omega, \cA, (\cF_t)_{t \in [0,T]}, P)$ be a complete filtered probability space satisfying the usual conditions, which is to say that $\cF_0$ contains all $P$-null sets and $\cF_t = \cap_{s > t} \cF_s$ for all $t \in [0,T]$. By $L^p(\Omega,H)$, $p \in [1, \infty)$
	we denote the space of all $H$-valued random variables $X$ with norm $\norm{X}{L^p(\Omega,H)} = (\E[\norm{X}{H}^p])^{1/p}$. Let $W\colon \Omega \times [0,T] \to H$ be a Wiener
	process with a covariance operator $Q \in \cL(H)$ that is positive semidefinite and self-adjoint, but not necessarily of trace-class.
	As is usual in this setting, we write $H_0 = Q^{\frac{1}{2}}(H)$, which is a Hilbert space when equipped with the inner product $\inpro[H_0]{\cdot}{\cdot} = \inpro[H]{Q^{-\frac{1}{2}} \cdot}{Q^{-\frac{1}{2}} \cdot}$, where $Q^{-\frac{1}{2}}$ is the pseudo-inverse of $Q^{\frac{1}{2}}$. Note that for $\Gamma_1, \Gamma_2 \in \cL^0_2 = \cL_2(H_0,H)$, 
	\begin{equation}
	\label{eq:HSvTrace}
	\left|\inpro[\cL^0_2]{\Gamma_1}{\Gamma_2}\right| = \left|\trace(\Gamma_1 Q \Gamma^*_2)\right| \le \norm{\Gamma_1 Q \Gamma^*_2}{\trace}
	\end{equation}
	whenever the right hand side is finite. Here and below the shorthand notation $\norm{\cdot}{\trace} = \norm{\cdot}{\cL_1(H)}$ is used. The Wiener process allows us to handle It\^o integrals $\int_{0}^{T} \Phi(t) \dd W(t)$, for predictable stochastic processes $\Phi\colon [0,T]\times\Omega \to \cL^0_2$. The following Burkholder--Davis--Gundy type inequality turns out to be useful.
	\begin{lemma}[{\cite[Lemma~7.2]{DPZ92}}]
		\label{lem:BDG}
		For any $p \in [1,\infty)$, there exists a constant $C > 0$, such that for any predictable stochastic process $\Phi\colon [0,T]\times\Omega \to \cL^0_2$ with $\norm{\Phi}{L^p(\Omega,L^2([0,T], \cL_2^0)} < \infty$,
		\begin{equation*}
		\Bignorm{\int_{0}^{T} \Phi(t) \dd W(t)}{L^p(\Omega,H)} \le C  \norm{\Phi}{L^p(\Omega,L^2([0,T], \cL_2^0))}.
		\end{equation*}
	\end{lemma} 
	
	We are now ready to introduce the equation studied in this paper, 
	\begin{equation}
	\label{eq:original_wave_equation}
	\begin{dcases*}
	\dd \dot{u} (t) + \Lambda u(t) \dd t =  F(t,u(t)) \dd t + G(t) \dd W(t), t \in (0,T],  \\ 
	u(0) = u_0, \dot{u}(0) = v_0.
	\end{dcases*}
	\end{equation}
	Here the solution process $u$ and the Wiener process $W$ take values in the Hilbert space $H$, $\dot{u}$ denotes the time derivative of $u$ and $F$ and $G$ are deterministic mappings. The operator $\Lambda$ is a densely defined, linear, unbounded positive self-adjoint operator with compact inverse, implying that it has an orthonormal eigenbasis $(e_j)_{j=1}^\infty$ spanning $H$ with an increasing sequence $(\lambda_j)_{j=1}^\infty$ of strictly positive eigenvalues, which are used to define fractional powers $\Lambda^{\frac{\alpha}{2}}$, $\alpha \in \R$ (see~\cite[Appendix~B]{K14}). We adopt the notation $\dot{H}^\alpha$ for the Hilbert space $D(\Lambda^{\frac{\alpha}{2}})$ and remark that 
	$\dot{H}^{-\alpha} \simeq (\dot{H}^\alpha)^*$ for $\alpha \ge 0$, where $(\dot{H}^\alpha)^*$ is the dual of $\dot{H}^\alpha$ with respect to $\inpro[\dot{H}^0]{\cdot}{\cdot}$ and $\dot{H}^0$ is identified with $(\dot{H}^0)^*$ by the Riesz representation theorem. We have that $\dot{H}^0 = H$ and that $\dot{H}^\zeta \subset \dot{H}^\alpha$ for $\alpha \le \zeta \in \R$ , where the embedding is dense and continuous. By \cite[Lemma~2.1]{BKM18}, for every $\zeta \in \R$,  $\Lambda^\frac{\alpha}{2}$ can be uniquely extended to an operator in $\cL(\dot{H}^\zeta,\dot{H}^{\zeta-\alpha})$. We make no notational distinction between $\Lambda^\frac{\alpha}{2}$ and its extension.
	
	In order to treat~\eqref{eq:original_wave_equation} in a semigroup framework, we define for $\alpha \in \R$ the Hilbert space $\cH^\alpha = \dot{H}^\alpha \oplus \dot{H}^{\alpha-1}$ with inner product $\inpro[\cH^\alpha]{v}{w} = \inpro[\dot{H}^\alpha]{v_1}{w_1} + \inpro[\dot{H}^{\alpha-1}]{v_2}{w_2}$ for $v = [v_1,v_2]^\top, w = [w_1,w_2]^\top \in \cH^\alpha$. Writing $\cH = \cH^0$, let $A\colon D(A) = \cH^1 \to \cH$, $B\colon \dot{H}^{-1} \to \cH$ 
	and $\Theta^\frac{\alpha}{2}\colon \cH^\alpha \to \cH, \alpha \in \R,$ be given by 
	\begin{equation*}
	A = \left[\begin{array}{cc}
	0 &- I \\
	\Lambda &0
	\end{array}\right],
	B = \left[
	\begin{array}{c}
	0 \\
	I
	\end{array}\right]
	\text{ and }
	\Theta^{\frac{\alpha}{2}} = \left[\begin{array}{cc}
	\Lambda^{\frac{\alpha}{2}} &0 \\
	0 & \Lambda^{\frac{\alpha}{2}}
	\end{array}\right].
	\end{equation*}
	The third operator is used to relate the norms of $\cH^\alpha$ and $\cH$ via $\norm{\cdot}{\cH^\alpha} = \norm{\Theta^{\frac{\alpha}{2}} \cdot}{\cH}$.  We also consider $P^1$, the projection onto the first coordinate of $\cH$, i.e., $P^1 v = v_1$ for $v = [v_1, v_2]^\top \in \cH$. Note that $\Theta^{\frac{\alpha}{2}} B = B \Lambda^{\frac{\alpha}{2}}$ and that therefore, the identities 
	\begin{equation}
	\label{eq:B_bound_1}
	\norm{\Theta^{\frac{\alpha}{2}} B v}{\cH} = \norm{\Lambda^{\frac{\alpha-1}{2}}v}{\dot{H}^0} = \norm{v}{\dot{H}^{\alpha-1}},
	\end{equation}
	with $v \in \dot{H}^{\alpha-1}$, and
	\begin{equation}
	\label{eq:B_bound_2}
	\norm{B}{\cL(\dot{H}^{-1},\cH)} = 
	\norm{B\Lambda^\frac{1}{2}}{\cL(\dot{H}^0,\cH)} = \norm{\Theta^{\frac{1}{2}}B}{\cL(\dot{H}^0,\cH)} = 1,
	\end{equation}
	hold. 
	
	The operator $-A$ is the generator of a $C_0$-semigroup (actually a group, see~\cite{L12}) which, for $t \in \R$, can be written as
	\begin{equation*}
	E(t) = 
	\left[\begin{array}{cc}
	C(t) & \Lambda^{-\frac{1}{2}} S(t) \\
	-\Lambda^{\frac{1}{2}} S(t) & C(t)
	\end{array}\right]
	= 
	\left[\begin{array}{cc}
	\cos(t \Lambda^{\frac{1}{2}}) & \Lambda^{-\frac{1}{2}} \sin(t \Lambda^{\frac{1}{2}}) \\
	-\Lambda^{\frac{1}{2}} \sin(t \Lambda^{\frac{1}{2}}) & \cos(t \Lambda^{\frac{1}{2}})
	\end{array}\right].
	\end{equation*}
	It fulfills
	\begin{equation}
	\label{eq:semigroup_bound} 
	\norm{E(t)}{\cL(\cH)} \le 1
	\end{equation}
	uniformly in $t \in \R$. We note the commutative properties, with $\alpha \in \R$,
	\begin{equation*}
	\Theta^{\frac{\alpha}{2}}E(\cdot) = E(\cdot) \Theta^{\frac{\alpha}{2}}
	\end{equation*}
	and 
	\begin{equation}
	\label{eq:P^1EB_Theta_commute}
	\Lambda^{\frac{\alpha}{2}}P^1 E(\cdot)B = P^1 E(\cdot)B \Lambda^{\frac{\alpha}{2}},
	\end{equation}
	so that, for all $\alpha \in [0,1]$ there exists by~\cite[Lemma~4.4]{KLL13} a constant $C>0$ such that, for all $t,s \in [0,T]$,
	\begin{equation}
	\label{eq:semigroup_time_regularity}
	\norm{\Theta^{-\frac{\alpha}{2}}\left(E(t)-E(s)\right)}{\cL(\cH)} = \norm{\left(E(t)-E(s)\right)\Theta^{-\frac{\alpha}{2}}}{\cL(\cH)} \le C |t-s|^{\alpha}
	\end{equation}
	and by~\eqref{eq:B_bound_2} and an argument similar to~\cite[(4.1)]{CLS13}, we have
	\begin{equation}
	\label{eq:sine_time_regularity}
	\begin{split}
	\norm{\Lambda^{\frac{\alpha}{2}}P^1 E(t-s)B}{\cL(\dot{H}^0)} &= \norm{P^1E(t-s)B\Lambda^{\frac{\alpha}{2}}}{\cL(\dot{H}^0)} \\
	&= \norm{S(t-s)\Lambda^{\frac{\alpha-1}{2}}}{\cL(\dot{H}^0)} \le C |t-s|^{1-\alpha}.
	\end{split}
	\end{equation}
	
	If we write $X(t) = [X_1(t), X_2(t)]^\top = [u(t), \dot{u}(t)]^\top$ for $t \in [0,T]$, then~\eqref{eq:original_wave_equation} can be written in the abstract It\^o form
	\begin{equation}
	\label{eq:ito_wave_equation}
	\dd X(t) + A X(t) \dd t = BF(t,X_1(t)) \dd t + B G(t) \dd W(t)
	\end{equation}
	with initial condition $X(0) = x_0 = [u_0, v_0]^\top$. Under the following assumption, \eqref{eq:ito_wave_equation} has a mild solution given by
	\begin{equation}
	\label{eq:mild_solution}
	X(t) = E(t) x_0 + \int^t_0 E(t-s) BF(s,X_1(s)) \dd s + \int^t_0 E(t-s) B G(s) \dd W(s)
	\end{equation}
	for $t \in [0,T]$, the existence of which we show below.
	\begin{assumption}
		There exist parameters $\beta, \eta, \delta \ge 0$ and $\theta \le \min(\beta,\delta,1)$
		and a constant $C>0$ such that the data in~\eqref{eq:ito_wave_equation} fulfills the following requirements. 
		\label{assumptions:1}
		\hfill
		\begin{enumerate}[label=(\roman*)]
			\item \label{assumptions:1:G} 
			The mapping $G \colon [0,T] \to \cL_2(H_0,\dot{H}^{\beta-1})$ satisfies
			\begin{align*}
			\norm{\Lambda^{\frac{\beta-1}{2}} \left(G(t_1)-G(t_2)\right)}{\cL_2^0} \le C |t_1-t_2|^{\eta}
			\end{align*}
			for all $t_1, t_2 \in [0,T]$ and $\norm{\Lambda^{\frac{\beta-1}{2}} G(t)}{\cL_2^0} \le C$ for some $t \in [0,T]$.
			\item \label{assumptions:1:F}
			The function $F\colon [0,T] \times \dot{H}^0 \to \dot{H}^{0}$ satisfies
			\begin{equation*}
			\norm{ \Lambda^{-\frac{1}{2}}\left( F(t,u) - F(t,v) \right)}{\dot{H}^0} \le C \norm{u - v}{\dot{H}^0}
			\end{equation*}
			for all $t \in [0,T]$ and $u,v \in \dot{H}^0$, 
			\begin{equation*}
			\norm{\Lambda^{\frac{\alpha}{2}} F(t,u)}{\dot{H}^0} \le C \left(1 + \norm{\Lambda^{\frac{\alpha}{2}}u}{\dot{H}^0}\right),
			\end{equation*} for all $t \in [0,T]$, $u \in \dot{H}^\alpha$ and $\alpha \in \{0,\theta\}$ and
			\begin{equation*}
			\norm{\Lambda^{-\frac{1}{2}}\left(  F(t,u) - F(s,u) \right)}{\dot{H}^0} \le C \left(1 +\norm{u}{\dot{H}^0}\right) |t-s|^\eta
			\end{equation*}
			for all $s, t \in [0,T]$ and $u \in \dot{H}^0$.
			\item  \label{assumptions:1:x0} The initial value~$x_0 \in \cH^{\delta}$ is deterministic.
		\end{enumerate}
	\end{assumption}
	
	The following theorem is very similar to, e.g., \cite{W15}, but since the mappings $F$ and $G$ depend on $t$, and the assumptions on $F$ are slightly different than those in \cite{W15}, we include a proof of our own. 
	
	\begin{theorem}
		\label{thm:X_spat_reg}
		Let Assumption~\ref{assumptions:1} be satisfied. Then~\eqref{eq:ito_wave_equation} has a unique mild solution given by~\eqref{eq:mild_solution} and for any $r \le \min(\beta,\delta,1+\theta)$, $p \in [1,\infty)$,
		\begin{equation}
		\label{eq:spat_reg_X}
		\sup_{t \in [0,T]} \norm{X(t)}{L^p(\Omega,\cH^r)} < \infty.
		\end{equation}
	\end{theorem}
	\begin{proof}
		Let $t \in [0,T]$ be fixed. Using the fact that $\Lambda^{-\alpha} \in \cL(\dot{H}^0)$ for any $\alpha \ge 0$, we have by Assumption~\ref{assumptions:1}\ref{assumptions:1:F} and~\eqref{eq:B_bound_1} that for any $x=[x_1, x_2]^\top, y=[y_1, y_2]^\top \in \cH$,
		\begin{equation*}
		\norm{BF(t,x_1)-BF(t,y_1)}{\cH} = \norm{\Lambda^{-\frac{1}{2}} \left(F(t,x_1)-F(t,y_1)\right)}{\dot{H}^0} \lesssim \norm{x - y}{\cH}.
		\end{equation*}
		Similarly, recalling also~\eqref{eq:schatten_bound_1}, 
		\begin{align*}
		\norm{BF(t,x_1)}{\cH} + \norm{BG(t)}{\cL_2(H_0,\cH)} &= \norm{\Lambda^{-\frac{1}{2}}F(t,x_1)}{\dot{H}^0} + \norm{\Lambda^{-\frac{1}{2}}G(t)}{\cL_2^0} \\
		&\lesssim \norm{F(t,x_1)}{\dot{H}^0} +
		\norm{\Lambda^{\frac{\beta-1}{2}} G(t)}{\cL_2^0}   \\ 
		&\lesssim 1 + \norm{x_1}{\dot{H}^0} \le 1 + \norm{x}{\cH}.
		\end{align*}	
		
		The existence and uniqueness of the mild solution~\eqref{eq:mild_solution} now follows from~\cite[Theorem 7.2]{DPZ14} (for $p \ge 2$ and clearly also for $p \in [1,2)$ since $(\Omega,\cA,(\cF_t)_{t \in [0,T]},P)$ is a probability space), which also guarantees that~\eqref{eq:spat_reg_X} holds for $r=0$. The case $r<0$ follows immediately. To show~\eqref{eq:spat_reg_X} for $0<r\le\min(\beta,\delta,1)$, we first note that 
		\begin{align*}
		\norm{X(t)}{L^p(\Omega,\cH^r)} &= \norm{\Theta^{\frac{r}{2}} X(t)}{L^p(\Omega,\cH)} \\ 
		&\le \norm{E(t)\Theta^{\frac{r}{2}}x_0}{\cH} + \int^t_0 \norm{E(t-s) \Theta^{\frac{r}{2}} B F(s,X_1(s))}{L^p(\Omega,\cH)} \dd s \\&\quad+ \lrnorm{\int^t_0 E(t-s) \Theta^{\frac{r}{2}}B G(s) \dd W(s)}{L^p(\Omega,\cH)}.
		\end{align*}
		For the first term, \eqref{eq:semigroup_bound} and Assumption~\ref{assumptions:1}\ref{assumptions:1:x0} imply
		\begin{equation*}
		\norm{E(t)\Theta^{\frac{r}{2}}x_0}{\cH} \le \norm{x_0}{\cH^r}  \lesssim 1.
		\end{equation*}
		Next, we first note that since $r\le1$, by~\eqref{eq:semigroup_bound}, \eqref{eq:B_bound_1}, Assumption~\ref{assumptions:1}\ref{assumptions:1:F} and~\eqref{eq:spat_reg_X} with $r=0$,
		\begin{equation}
		\label{eq:X_spat_reg_F_bound}
		\begin{split}
		&\int^t_0 \norm{E(t-s) \Theta^{\frac{r}{2}} B F(s,X_1(s))}{L^p(\Omega,\cH)} \dd s \\
		&\quad\le t \sup_{s \in [0,T]} \norm{\Theta^{\frac{r}{2}} B F(s,X_1(s))}{L^p(\Omega,\cH)} = t \sup_{s \in [0,T]} \norm{\Lambda^{\frac{r-1}{2}} F(s,X_1(s))}{L^p(\Omega,\dot{H}^0)} \\
		&\quad\lesssim \sup_{s \in [0,T]} \norm{ F(s,X_1(s))}{L^p(\Omega,\dot{H}^0)} \lesssim 1 + \sup_{s \in [0,T]} \norm{X_1(s)}{L^p(\Omega,\dot{H}^0)} \\
		&\quad\le 1 + \sup_{s \in [0,T]} \norm{X(s)}{L^p(\Omega,\cH)} \lesssim 1.
		\end{split}
		\end{equation}
		For the third term, by analogous arguments, Assumption~\ref{assumptions:1}\ref{assumptions:1:G} and Lemma~\ref{lem:BDG} (note that the integrand below is deterministic),
		\begin{align*}
		&\lrnorm{\int^t_0 E(t-s) \Theta^{\frac{r}{2}}B G(s) \dd W(s)}{L^p(\Omega,\cH)} \\
		&\quad\lesssim \left(\int^t_0 \norm{E(t-s) \Theta^{\frac{r}{2}}B G(s)}{\cL_2(H_0,\cH)}^2 \dd s\right)^{\frac{1}{2}} \lesssim  t^{\frac{1}{2}} \sup_{s \in [0,T]} \norm{E(t-s) \Theta^{\frac{r}{2}}B G(s)}{\cL_2(H_0,\cH)} \\
		&\quad\lesssim \sup_{s \in [0,T]} \norm{\Lambda^{\frac{r-1}{2}} G(s)}{\cL_2^0} \lesssim 1.
		\end{align*} 
		Altogether, this shows~\eqref{eq:spat_reg_X} for $0\le r\le\min(\beta,\delta,1)$. Finally, for the case $r \in (1,\min(\beta,\delta,1+\theta)]$ we repeat the arguments above, replacing the calculation in~\eqref{eq:X_spat_reg_F_bound} with
		\begin{align*}
		&\sup_{s \in [0,T]} \norm{\Lambda^{\frac{r-1}{2}} F(s,X_1(s))}{L^p(\Omega,\dot{H}^0)}\\
		&\quad
		\lesssim
		\sup_{s \in [0,T]} \norm{\Lambda^{\frac{\theta}{2}} F(s,X_1(s))}{L^p(\Omega,\dot{H}^0)}
		\lesssim 1 + \sup_{s \in [0,T]} \norm{\Lambda^{\frac{\theta}{2}}X_1(s)}{L^p(\Omega,\dot{H}^0)} \\
		&\quad\le 1 + \sup_{s \in [0,T]} \norm{\Theta^{\frac{\min(\beta,\delta,1)}{2}}X(s)}{L^p(\Omega,\cH)},
		\end{align*}
		which is finite since we have shown that~\eqref{eq:spat_reg_X} holds with $r=\min(\beta,\delta,1)$
		and by assumption $\theta \le \min(\beta,\delta,1)$.
	\end{proof}
	
	From here on we denote by $r = \min(\beta,\delta,1+\theta)$ the maximum spatial regularity of the solution to~\eqref{eq:ito_wave_equation}. A temporal regularity result finishes this section of the paper.
	
	\begin{theorem}
		\label{thm:mild_holder}
		Let Assumption~\ref{assumptions:1} be satisfied and let $r = \min(\beta,\delta,1+\theta)$. Then, for all $\alpha \le r$, $p \ge 1$, there exists a positive constant $C$ such that for all $s,t \in [0,T]$, 
		\begin{equation}
		\label{eq:mild_holder_full}
		\norm{X(t)-X(s)}{L^p(\Omega,\cH^{\alpha})} \le C |t-s|^{\min(r-\alpha,\frac{1}{2})}.
		\end{equation}
		and
		\begin{equation}
		\label{eq:mild_holder_x1}
		\norm{X_1(t)-X_1(s)}{L^p(\Omega,\dot{H}^{\alpha})} \le C |t-s|^{\min(r-\alpha,1)}.
		\end{equation}
	\end{theorem}
	\begin{proof}
		Fix $0 \le s \le t \le T$. We first note that 
		\begin{align*}
		\Theta^{\frac{\alpha}{2}} \left(X(t) - X(s) \right) &= \Theta^{\frac{\alpha}{2}} (E(t-s) - I) X(s) + \int^t_s \Theta^{\frac{\alpha}{2}} E(t-r) BF(r,X_1(r)) \dd r \\
		&\quad+ \int^t_s \Theta^{\frac{\alpha}{2}} E(t-r) B G(r) \dd W(r),
		\end{align*}
		so that therefore
		\begin{equation}
		\label{eq:mild_holder_error_split}
		\begin{split}
		\norm{X(t) - X(s)}{L^p(\Omega,\cH^\alpha)}
		&= \norm{\Theta^{\frac{\alpha}{2}} \left(X(t) - X(s) \right)}{L^p(\Omega,\cH)} \\ &\le \norm{\Theta^{\frac{\alpha}{2}} (E(t-s) - I) X(s)}{L^p(\Omega,\cH)} 
		\\ &\quad+ \lrnorm{\int^t_s \Theta^{\frac{\alpha}{2}} E(t-r) BF(r,X_1(r)) \dd r}{L^p(\Omega,\cH)} \\ 
		&\quad+ \lrnorm{\int^t_s \Theta^{\frac{\alpha}{2}} E(t-r) B G(r) \dd W(r)}{L^p(\Omega,\cH)}.
		\end{split}
		\end{equation}
		By~\eqref{eq:semigroup_time_regularity} and Theorem~\ref{thm:X_spat_reg} the first term on the right hand side of~\eqref{eq:mild_holder_error_split} is bounded by
		\begin{equation*}
		\norm{\Theta^{\frac{\alpha-r}{2}}(E(t-s) - I)}{\cL(\cH)} \norm{\Theta^{\frac{r}{2}} X(s)}{L^p(\Omega,\cH)} \lesssim (t-s)^{\min(r-\alpha,1)}.
		\end{equation*}
		For the second term, we have, since $\alpha \le r \le 1+\theta$, by~\eqref{eq:B_bound_1},~\eqref{eq:semigroup_bound},  Assumption~\ref{assumptions:1}\ref{assumptions:1:F} and Theorem~\ref{thm:X_spat_reg},
		\begin{align*}
		&\lrnorm{\int^t_s \Theta^{\frac{\alpha}{2}} E(t-r) BF(r,X_1(r)) \dd r}{L^p(\Omega,\cH)} \\
		&\quad\le \int^t_s \norm{E(t-r) \Theta^{\frac{\alpha}{2}} BF(r,X_1(r))}{L^p(\Omega,\cH)} \dd r \\
		&\quad\le \int^t_s \norm{E(t-r) }{\cL(\cH)} \norm{\Lambda^{\frac{\alpha-1}{2}} F(r,X_1(r))}{L^p(\Omega,\dot{H}^0)}\dd r \\
		&\quad\lesssim (t-s) \sup_{r \in [0,T]} \norm{\Lambda^{\frac{\theta}{2}} F(r,X_1(r))}{L^p(\Omega,\dot{H}^0)}  \lesssim \left(1 + \sup_{r \in [0,T]} \norm{\Lambda^{\frac{\theta}{2}} X_1(r)}{L^p(\Omega,\dot{H}^0)} \right) (t-s).
		\end{align*}
		Similarly, Lemma~\ref{lem:BDG} yields that the third term on the right hand side of~\eqref{eq:mild_holder_error_split} is bounded by a constant times
		\begin{align*}
		&\left( \int^t_s \norm{E(t-r)\Theta^{\frac{\alpha}{2}}BG(r)}{\cL_2(H_0,\cH)}^2 \dd r \right)^{\frac{1}{2}} \\
		&\quad\le (t-s)^{\frac{1}{2}} \sup_{r \in [0,T]} \norm{\Theta^{\frac{\alpha}{2}}BG(r)}{\cL_2(H_0,\cH)} = (t-s)^{\frac{1}{2}} \sup_{r \in [0,T]} \norm{\Lambda^{\frac{\alpha-1}{2}}G(r)}{\cL_2^0} \\
		&\quad\lesssim (t-s)^{\frac{1}{2}} \sup_{r \in [0,T]} \norm{\Lambda^{\frac{\beta-1}{2}} G(r)}{\cL_2^0} \lesssim (t-s)^{\frac{1}{2}},
		\end{align*}
		which completes the proof of~\eqref{eq:mild_holder_full}. The proof of~\eqref{eq:mild_holder_x1} is entirely similar, except for the analysis of the stochastic term. By~\eqref{eq:sine_time_regularity}, it satisfies
		\begin{align*}
		&\left( \int^t_s \norm{P^1 E(t-r)B\Lambda^{\frac{\alpha}{2}}G(r)}{\cL_2^0}^2 \dd r \right)^{\frac{1}{2}} \\
		&\quad= \left( \int^t_s \norm{P^1 E(t-r)B\Lambda^{\frac{1+\alpha-\beta}{2}}\Lambda^{\frac{\beta-1}{2}}G(r)}{\cL_2^0}^2 \dd r \right)^{\frac{1}{2}} \\
		&\quad\lesssim \sup_{\tau \in [0,T]} \norm{\Lambda^{\frac{\beta-1}{2}}G(\tau)}{\cL_2^0} \left( \int^t_s (t-r)^{2 \min(\beta - \alpha,1)} \dd r \right)^{\frac{1}{2}} \lesssim (t-s)^{ \min(\beta - \alpha,1)+\frac{1}{2}}
		\end{align*}
		which combined with the previous estimates proves~\eqref{eq:mild_holder_x1} and therefore finishes the proof of the theorem.
	\end{proof}
	
	\section{Approximation and convergence}
	\label{sec:appr_and_conv}
	
	We now consider a more concrete setting by taking $H = \dot{H}^0 = L^2(\cD)$, where $\cD$ denotes a convex polygonal bounded domain in $\R^d$. Let $\Lambda = - \Delta$ be the Laplace operator on $H$ with zero Dirichlet boundary conditions. With this, the spaces $(\dot{H}^{\alpha})_{\alpha \in \R}$ are related to classical Sobolev spaces by $\dot{H}^0 = L^2(\cD),$ $\dot{H}^1 = H^1_0(\cD)$ and $\dot{H}^2 = D(\Lambda) = H^2(\cD) \cap H^1_0(\cD)$, where $H^n(\cD)$ denotes the Sobolev space of order $n \in \N$ on $\cD$ and $H_0^1(\cD)$ is the subspace of functions in $H^1(\cD)$ that are zero on the boundary of $\cD$ (see also Appendix~\ref{sec:appendix}). Next, we introduce our fully discrete approximation of the solution to~\eqref{eq:ito_wave_equation}. For the spatial discretization, a standard continuous finite element method is employed and for the temporal discretization, a rational approximation of the semigroup. This is the same approach as in~\cite[Section 5]{KLS15} to which the reader is referred for further details, but see also~\cite{KLL12} and~\cite{KLL13}. We then show a strong and a weak convergence result for this approximation.
	
	To be precise, we take $(V^\kappa_h)_{h \in (0,1]} \subset \dot{H}^1$, $\kappa \in \{2,3\}$, to be a standard family of finite element function spaces consisting of continuous piecewise polynomials of degree $\kappa-1$, with respect to a regular family of triangulations of $\cD$ with maximal mesh size $h$, that are zero on the boundary of $\cD$. They are equipped with the inner product $\inpro[V_h^\kappa]{\cdot}{\cdot} = \inpro[\dot{H}^0]{\cdot}{\cdot}$. On this space, let a discrete counterpart $\Lambda_h\colon V^\kappa_h \to V^\kappa_h$ to $\Lambda$ be defined by 
	\begin{equation*}
	\inpro[\dot{H}^0]{\Lambda_h v_h}{u_h} = \inpro[\dot{H}^0]{\Lambda^{\frac{1}{2}} v_h}{\Lambda^{\frac{1}{2}}u_h} = \inpro[\dot{H}^1]{v_h}{u_h}
	\end{equation*}
	for all $v_h, u_h \in V^\kappa_h$. Fractional powers of $\Lambda_h$ are defined in the same way as for $\Lambda$. We define the generalized orthogonal projector $P_h \colon \dot{H}^{-1} \to V^\kappa_h$ by $\inpro{P_h v}{v_h} = {_{\dot{H}^{-1}}\langle}
	v, v_h \rangle_{\dot{H}^{1}}$ for all $v \in \dot{H}^{-1}$ and $v_h \in V^\kappa_h$, where $_{\dot{H}^{-1}}\langle
	\cdot, \cdot \rangle_{\dot{H}^{1}}$ denotes the dual pairing with respect to $\dot{H}^0$. Note that $P_h$ coincides with the usual orthogonal projector when restricted to $\dot{H}^0$. For our convergence results, we need the following assumption on $(V_h^\kappa)_{h\in(0,1]}$.
	\begin{assumption}
		\label{assumptions:FEM}
		There exists a constant $C > 0$ such that, for all $h \in (0,1]$, the operators $\Lambda_h$ and $P_h$ satisfy
		\begin{align}
		\label{eq:inv_ineq}
		\big\|\Lambda_h^{\frac \alpha 2}P_h\big\|_{\cL(\dot{H}^0)}
		&\leq
		C h^{-\alpha},
		\quad
		\alpha \in [0,2], 
		\\
		\label{eq:Lambda_norm_equivalence_1}
		\big\|
		\Lambda_h^{\frac \alpha 2}P_h v
		\big\|_{\dot{H}^0}
		&\leq
		C
		\big\|
		\Lambda^{\frac \alpha 2}v
		\big\|_{\dot{H}^0},
		\quad
		v \in \dot H^{\alpha}, \alpha \in [-1,1]. 
		\end{align}
	\end{assumption} 
	In our setting, this assumption is fulfilled if the mesh underlying $V_h^\kappa$ is quasi-uniform, see~\cite[(3.28)]{T06} and~\cite[(3.17)]{K14}. The counterpart to this assumption, that there exists a constant $C > 0$ such that
	\begin{equation}
	\label{eq:Lambda_norm_equivalence_2}
	\norm{\Lambda^\frac{\alpha}{2}v_h}{\dot{H}^0} \le C \norm{\Lambda_h^\frac{\alpha}{2}v_h}{\dot{H}^0},
	\quad
	v_h \in V_h^\kappa, \alpha \in [-1,1],
	\end{equation}
	holds without the assumption of quasi-uniformity. A combination of ~\eqref{eq:inv_ineq} and~\eqref{eq:Lambda_norm_equivalence_1} yields, for $v \in \dot{H}^0$ and $\alpha \in [0,1]$, 
	\begin{equation}
	\label{eq:Lambda_h_inv_ineq}
	\norm{P_h \Lambda^{\frac{\alpha}{2}} v}{\dot{H}^0} = \norm{\Lambda_h^{\frac{\alpha}{2}} \Lambda_h^{-\frac{\alpha}{2}}P_h \Lambda^{\frac{\alpha}{2}} v}{\dot{H}^0} \le C h^{-\alpha} \norm{\Lambda_h^{-\frac{\alpha}{2}}P_h \Lambda^{\frac{\alpha}{2}} v}{\dot{H}^0} \lesssim h^{-\alpha} \norm{v}{\dot{H}^0},
	\end{equation}
	where we have used the fact that $\Lambda^{\frac{\alpha}{2}} v \in \dot{H}^{-\alpha}$. 
	
	Let
	\begin{equation*}
	A_h = \left[\begin{array}{cc}
	0 &- I \\
	\Lambda_h &0
	\end{array}\right]
	\end{equation*}
	be a discrete counterpart to $A$ on the product space $\cV_h^\kappa = V^\kappa_h \oplus V^\kappa_h$, equipped with the same inner product as $\cH$. With some abuse of notation, by the expression $P_h v$, $v = [v_1, v_2]^\top \in \cH$, we denote the element $[P_h v_1, P_h v_2]^\top \in \cV_h^\kappa$. The operator $-A_h$, similarly to $-A$, generates a  $C_0$-(semi)group $E_h\colon \R \to \cL(\cV_h^\kappa)$. 
	
	Let
	\begin{equation*}
	\Theta_h^{\frac{\alpha}{2}} = \left[\begin{array}{cc}
	\Lambda_h^{\frac{\alpha}{2}} &0 \\
	0 & \Lambda_h^{\frac{\alpha}{2}}
	\end{array}\right] \in \cL(\cV_h^\kappa)
	\end{equation*}
	and note that, as a straightforward consequence of~\eqref{eq:Lambda_norm_equivalence_1} and~\eqref{eq:Lambda_norm_equivalence_2}, for every $\alpha \in [0,1]$ there exists a constant $C > 0$ such that for all $h \in (0,1]$, $v_h \in \cV^\kappa_h$ and $v \in \cH^\alpha$,
	\begin{equation}
	\label{eq:Theta_norm_equivalence}
	\begin{split}
	\norm{\Theta^\frac{\alpha}{2}v_h}{\cH} &\le C \norm{\Theta_h^\frac{\alpha}{2}v_h}{\cH}, \\
	\norm{\Theta_h^\frac{\alpha}{2} P_h v}{\cH} &\le C \norm{\Theta^\frac{\alpha}{2}v}{\cH},
	\end{split}
	\end{equation}
	and, using~\eqref{eq:Lambda_h_inv_ineq}, one shows that there exists a constant $C>0$ such that for all $h \in (0,1]$ and $v \in \cH$,
	\begin{equation}
	\label{eq:Theta_h_inv_ineq}
	\norm{P_h \Theta^{\frac{\alpha}{2}} v}{\cH} \le C h^{-\alpha}\norm{v}{\cH}.
	\end{equation}
	
	For the temporal discretization, consider a uniform time grid $t_j = j \Delta t = j (T/N_{\Delta t})$, $j=0, \ldots, N_{\Delta t}$. Let $R\colon \C \to \C$ be a rational function such that $|R(iy)|\le1$ for all $y \in \R$ and, for some $\rho \in \N$ and $C,b > 0$, $|R(iy)-e^{-iy}| \le C |y|^{\rho+1}$ for all $y \in \R$ with $|y|\le b$, where $i = \sqrt{-1}$. Several classes of such functions are described in~\cite[Section~4]{BB79}. They include certain Pad\'e approximations of the exponential function, for example, the one corresponding to the backward Euler scheme ($R(z)=1/(1+z)$) with order $\rho = 1$ and the one corresponding to the Crank--Nicolson scheme ($R(z)=(1-z/2)/(1+z/2)$) with order $\rho=2$. We write $E^n_{h,\Delta t} = R(\Delta t A_h)^n$ for the rational approximation of the operator $E_h(t_n)$. The Crank--Nicolson approximation given by
		\begin{equation*}
		R(\Delta t A_h) = \frac{1 - \frac{1}{2}\Delta t A_h }{1 + \frac{1}{2}\Delta t A_h }
		\end{equation*} is of particular importance for the wave equation as it preserves the energy. We define the interpolation $\tilde{E}\colon [0,T] \to \cL(\cH)$ of the approximation by the step function
	\begin{equation}
	\label{eq:semigroup_approximation_interpolation}
	\tilde{E}(t) = \chi_{\{0\}}(t) P_h + \sum_{j=1}^{N_{\Delta t}} \chi_{(t_{j-1},t_j]}(t) E^j_{h,\Delta t} P_h,
	\end{equation}
	for $t \in [0,T]$, where $\chi$ denotes the indicator function. For this interpolation, the stability result
	\begin{equation}
	\label{eq:fully_discrete_stability}
	\norm{\tilde{E}(t)}{\cL(\cH)} \le 1
	\end{equation} 
	for all $t \in [0,T]$ holds uniformly in $h, \Delta t \in (0,1]$, see, e.g., \cite{KLL13}. Moreover, the following error estimate with respect to its first component holds. 
	\begin{lemma}[{\cite[Lemma 5.2]{KLS15}}]
		\label{lem:semigroup_error}
		Let $\alpha \ge 0$ and assume that $\tilde{E}$ is given by~\eqref{eq:semigroup_approximation_interpolation} for a rational approximation $E_{h,\Delta t}$ of order $\rho \in \N$ of $E_h(\Delta t)$ for $\Delta t \le 1$. Then there exists a constant $C>0$ such that, for all $h, \Delta t \in (0,1]$,
		\begin{align*}
		&\sup_{t \in [0,T]} \left( \norm{P^1(\tilde{E}(t)-E(t))\Theta^{-\frac{\alpha}{2}}}{\cL(\cH,\dot{H}^0)} +  \norm{\Lambda^{-\frac{\alpha}{4}}P^1(\tilde{E}(t)-E(t))B\Lambda^{\frac{1}{2}-\frac{\alpha}{4}}}{\cL(\dot{H}^0)} \right) \\
		&\hspace{35mm}\le C \left(h^{\min(\alpha \frac{\kappa}{\kappa+1},\kappa)} + {\Delta t}^{\min(\alpha \frac{\rho}{\rho+1},1)} \right).
		\end{align*}
	\end{lemma}
	
	We now define the fully discrete approximation $(X^j_{h,\Delta t})_{j=0}^{N_{\Delta t}}$ by the recursion scheme
	\begin{equation}
	\label{eq:fully_discrete_scheme}
	X^j_{h, \Delta t} = E_{h,\Delta t} \left( X_{h, \Delta t}^{j-1}  + P_h B F(t_{j-1},P^1 X_{h, \Delta t}^{j-1}) \Delta t + P_h B G(t_{j-1}) \Delta W^j \right)
	\end{equation}
	for $j=1, 2, \ldots, N_{\Delta t}$, where $\Delta W^j = W(t_j)-W(t_{j-1})$, with $X^0_{h, \Delta t} = P_h x_0$. 
	In closed form it is given by the discrete mild solution formulation 
	\begin{equation*}
	X^n_{h, \Delta t} = E_{h,\Delta t}^n x_0 + \Delta t \sum_{j=0}^{n-1} E_{h,\Delta t}^{n-j} P_h B F(t_{j},P^1 X_{h, \Delta t}^{j}) + \sum_{j=0}^{n-1} E_{h,\Delta t}^{n-j} P_h B G(t_{j}) \Delta W^j
	\end{equation*}
	for $n = 0,1,\ldots,N_{\Delta t}$. 
	This we extend to a continuous time process $\tilde{X}\colon \Omega \times [0,T] \to \cV_h^\kappa$ by 
	\begin{equation}
	\label{eq:fully_discrete_mild_solution}
	\tilde{X}(t) = \tilde{E}(t)x_0 + \int^t_0 \tilde{E}(t-s)BF(\dfloor{s}, \tilde{X}_1(\dfloor{s})) \dd s+ \int^t_0 \tilde{E}(t-s)B G(\dfloor{s}) \dd W(s),
	\end{equation}
	for $t \in [0,T]$. Here $\tilde{X}_1 = P^1 \tilde{X}$ and $\dfloor{\cdot} = \floor{\cdot/\Delta t} \Delta t$, where $\floor \cdot$ denotes the floor function. It is straightforward to see that $\tilde{X}(t_j) = X^j_{h, \Delta t}$ $P$-a.s. for all $j=0, 1, \ldots, N_{\Delta t}$.
	
	\subsection{Strong convergence}
	
	Under Assumptions~\ref{assumptions:1} and~\ref{assumptions:FEM}, we now deduce a strong convergence result, i.e., convergence measured in $\norm{\cdot}{L^2(\Omega,\dot{H}^0)}$. For the weak convergence we shall also need strong convergence in a negative norm, for which we need an additional assumption. 
	
	\begin{assumption}
		\label{assumptions:2}
		In addition to the requirements of Assumption~\ref{assumptions:1}, there exist parameters $\mu\in [0,2]$, $\nu \in [\max(\mu - 1,0),\min(r,1)]$ and a constant $C > 0$ such that for every $t \in [0,T]$, $F(t, \cdot) \in \cG_\mathrm{p}^1(\dot{H}^0,\dot{H}^{-\min(\mu,1)})$ and  
		\begin{equation}
		\label{eq:ass2:munu}
		\norm{\Lambda^{-\frac{\mu}{2}} F'(t,u)v}{\dot{H}^0} \le C \left(1+\norm{\Lambda^{\frac{\nu}{2}}u}{\dot{H}^0}\right) \norm{\Lambda^{-\frac{\nu}{2}}v}{\dot{H}^0}
		\end{equation}
		for all $u \in \dot{H}^\nu$ and $v \in \dot{H}^{-\nu}$.
	\end{assumption} 
	Note that, as a consequence of the Lipschitz condition of Assumption~\ref{assumptions:1}\ref{assumptions:1:F} and the fact that $\norm{\cdot}{\dot{H}^{-1}}$ is continuous on $\dot{H}^{-\min(\mu,1)}$, we obtain that for all $t \in [0,T]$ and $u,v \in \dot{H}^0$,
	\begin{equation}
	\label{eq:ass2:derivativebounded}
	\begin{split}
	\norm{\Lambda^{-\frac{1}{2}} F'(t,u)v}{\dot{H}^0} &= \Bignorm{\Lambda^{-\frac{1}{2}} \lim_{\epsilon\to 0} \frac{1}{\epsilon} \left(F(t,u+\epsilon v)-F(t,u)\right)}{\dot{H}^0} \\
	&= \lim_{\epsilon\to 0} \frac{1}{\epsilon} \Bignorm{\Lambda^{-\frac{1}{2}}   \left(F(t,u+\epsilon v)-F(t,u)\right)}{\dot{H}^0} \lesssim \lim_{\epsilon\to 0} \frac{1}{\epsilon} \epsilon \norm{v}{\dot{H}^0} = \norm{v}{\dot{H}^0},
	\end{split}
	\end{equation}
	which, since $\cG_\mathrm{p}^1(\dot{H}^0,\dot{H}^{-\min(\mu,1)}) \subset \cG_\mathrm{p}^1(\dot{H}^0,\dot{H}^{-1})$, implies that $F(t,\cdot) \in \cG^1_\mathrm{b}(\dot{H}^0,\dot{H}^{-1})$.
	
	We also need the following version of Gronwall's lemma, 
	see~\cite[2.2~(9)]{G75}. 
	
	\begin{lemma}[Gronwall's lemma]
		\label{lem:Gronwall}
		Let 
		$C > 0$ and let
		$(a_j)_{j=1}^\infty$, $(b_j)_{j=1}^\infty$ be nonnegative sequences. 
		If
		\begin{equation*}
		a_n \le C + \sum^{n-1}_{j=1} b_j a_j , 
		\end{equation*}
		for all $n \ge 1$ then 
		\begin{equation*}
		a_n \le C \exp \left( \sum^{n-1}_{j=1} b_j \right) 
		\end{equation*}
		for all $n \ge 1$.
	\end{lemma}
	
	Note that we, for a real-valued sequence $(a_j)_{j=1}^\infty$, use the convention $\sum^0_{j=1} a_j = 0$. With this lemma in place, we are ready to deduce a strong convergence result.
	
	\begin{theorem}[Strong convergence]
		\label{thm:strong_convergence}
		Let $X$ be the mild solution of the stochastic wave equation given by~\eqref{eq:mild_solution}, let $\tilde{X}$ be the fully discrete approximation given by~\eqref{eq:fully_discrete_mild_solution}, let Assumptions~\ref{assumptions:1} and~\ref{assumptions:FEM} hold.
		Then, for all $\alpha \in [0,\min(r,1)]$ and any $p \in [1,\infty)$, 
		\begin{equation}
		\label{eq:fully_discrete_spatial_regularity}
		\sup_{\substack{n \in \{0,1,\ldots,N_{\Delta t}\} \\ h, \Delta t \in (0,1]}} \norm{\tilde X(t_n)}{L^p(\Omega,\cH^\alpha)} < \infty,
		\end{equation}
		and there exists a constant $C>0$ such that for all $h, \Delta t \in (0,1]$
		\begin{equation}
		\label{eq:fully_discrete_strong_converence}
		\sup_{n \in \{0,1,\ldots,N_{\Delta t}\}} \norm{\tilde{X_1}(t_n)-X_1(t_n)}{L^p(\Omega,\dot{H}^{0})} \le C \big(h^{r \frac{\kappa}{\kappa+1}} + {\Delta t}^{\min(r \frac{\rho}{\rho+1},\eta,1)} \big).
		\end{equation}
		
		If, in addition to this, Assumption~\ref{assumptions:2} also holds, then for any $p \in [1,\infty)$, there exists a constant $C>0$ such that for all $h, \Delta t \in (0,1]$
		\begin{equation}
		\label{eq:fully_discrete_strong_converence_negative_norm}
		\sup_{n \in \{0,1,\ldots,N_{\Delta t}\}} \norm{\tilde{X_1}(t_n)-X_1(t_n)}{L^p(\Omega,\dot{H}^{-\nu})} \le C \big(h^{r' \frac{\kappa}{\kappa+1}} + {\Delta t}^{\min(r' \frac{\rho}{\rho+1},\eta,1)} \big),
		\end{equation}
		where $r' = \min(\max(2\nu,\beta),\max(2\nu,1+\theta),\delta)$.
	\end{theorem}
	\begin{proof}
		We start by showing~\eqref{eq:fully_discrete_spatial_regularity} using the representation~\eqref{eq:fully_discrete_mild_solution}. Take $\alpha \in [0,\min(r,1)]$. First note that since $\Theta^{\frac{\alpha}{2}}_h$ commutes with $\tilde{E}(t_n)$ for all $n \in \{0,1,\ldots,N_{\Delta t}\}$ and since $\alpha\le1$, by~\eqref{eq:Theta_norm_equivalence} and~\eqref{eq:fully_discrete_stability}, \begin{equation*}
		\norm{\Theta^\frac{\alpha}{2}\tilde{E}(t_n)v}{\cH} \lesssim 
		\norm{\Theta_h^\frac{\alpha}{2}\tilde{E}(t_n)v}{\cH} 
		= \norm{\tilde{E}(t_n)\Theta_h^\frac{\alpha}{2}P_h v}{\cH}
		\lesssim 
		\norm{\Theta_h^\frac{\alpha}{2}P_h v}{\cH} \lesssim \norm{\Theta^\frac{\alpha}{2}v}{\cH}
		\end{equation*}
		for all $n \in \{0,1,\ldots,N_{\Delta t}\}$ and $v \in \cH^{\alpha}$.
		Using this along with Lemma~\ref{lem:BDG},~\eqref{eq:B_bound_1} and finally Assumption~\ref{assumptions:1}, one obtains for $n = 1, 2, \ldots, N_{\Delta t}$ that
		\begin{align*}
		&\norm{\tilde X(t_n)}{L^p(\Omega,\cH^\alpha)} \\
		&\quad= \norm{\Theta^{\frac{\alpha}{2}} X^n_{h,\Delta t}}{L^p(\Omega,\cH)} \\
		&\quad \lesssim \norm{\Theta^{\frac{\alpha}{2}}\tilde{E}(t_n) x_0}{\cH} + \int^{t_n}_0 \norm{\Theta^{\frac{\alpha}{2}}\tilde{E}(t_n-s)BF(\dfloor{s}, \tilde{X}_1(\dfloor{s}))}{L^p(\Omega,\cH)} \dd s  \\ 
		&\quad\quad+ \left( \int^{t_n}_0 \norm{\Theta^{\frac{\alpha}{2}}\tilde{E}(t_n-s)B G(\dfloor{s})}{\cL_2(H_0,\cH)}^2 \dd s \right)^{\frac{1}{2}} 
		\\
		&\quad\lesssim \norm{\Theta^{\frac{\alpha}{2}}x_0}{\cH} + \sum_{j=0}^{n-1} \Delta t \norm{\Theta^{\frac{\alpha}{2}}BF(t_{j},P^1 X_{h, \Delta t}^{j})}{L^p(\Omega,\cH)} + \left( \sum_{j=0}^{n-1} \Delta t \norm{\Theta^{\frac{\alpha}{2}} B G(t_{j})}{\cL_2(H_0, \cH)}^2 \right)^{\frac{1}{2}}
		\\
		&\quad= \norm{x_0}{\cH^\alpha} + \sum_{j=0}^{n-1} \Delta t \norm{\Lambda^{\frac{\alpha-1}{2}} F(t_{j},P^1 X_{h, \Delta t}^{j})}{L^p(\Omega,\dot{H}^0)} +  \left( \sum_{j=0}^{n-1} \Delta t \norm{\Lambda^{\frac{\alpha-1}{2}} G(t_{j})}{\cL_2^0} \right)^{\frac{1}{2}}\\
		&\quad\lesssim \norm{x_0}{\cH^\delta} + \Delta t \sum_{j=0}^{n-1} \left(1 + \norm{\Lambda^{\frac{\alpha-1}{2}} P^1 X_{h, \Delta t}^{j}}{L^p(\Omega,\dot{H}^0)}\right) + \sqrt{T} \lesssim 1 +\Delta t \sum_{j=0}^{n-1} \norm{\tilde X(t_j)}{L^p(\Omega,\cH^\alpha)}.
		\end{align*} 
		An application of Lemma~\ref{lem:Gronwall} now yields~\eqref{eq:fully_discrete_spatial_regularity}. 
		
		We prove~\eqref{eq:fully_discrete_strong_converence} and~\eqref{eq:fully_discrete_strong_converence_negative_norm} in tandem, with $\alpha \in \{0,\nu\}$, by first making the split
		\begin{align*}
		&\norm{\tilde{X_1}(t_n)-X_1(t_n)}{L^p(\Omega,\dot{H}^{-\alpha})} \\
		&\quad= \norm{\Lambda^{-\frac{\alpha}{2}} P^1(\tilde{X}(t_n)-X(t_n))}{L^p(\Omega,\dot{H}^0)} \\
		&\quad\le \norm{\Lambda^{-\frac{\alpha}{2}}P^1(\tilde{E}(t_n)-E(t_n))x_0}{L^p(\Omega,\dot{H}^0)} \\
		&\quad\quad+ \int^{t_n}_0 \lrnorm{\Lambda^{-\frac{\alpha}{2}}P^1\left(\tilde{E}(t_n-s)BF(\dfloor{s}, \tilde{X}_1(\dfloor{s})) - E(t_n-s)BF(s, X_1(s))\right)}{L^p(\Omega,\dot{H}^0)} \dd s \\
		&\quad\quad+ \lrnorm{\int^{t_n}_0 \Lambda^{-\frac{\alpha}{2}} P^1\left(\tilde{E}(t_n-s)BG(\dfloor{s})-E(t_n-s)BG(s)\right) \dd W(s)}{L^p(\Omega,\dot{H}^0)} \\
		&\quad= \mathrm{I} + \mathrm{II} + \mathrm{III},
		\end{align*}
		for arbitrary $n = 1, 2, \ldots, N_{\Delta t}$. For the first term, by Lemma~\ref{lem:semigroup_error} and Assumption~\ref{assumptions:1}\ref{assumptions:1:x0},
		\begin{align*}
		\mathrm{I} &=
		\norm{\Lambda^{-\frac{\alpha}{2}}P^1(\tilde{E}(t_n)-E(t_n))\Theta^{-\frac{\delta}{2}}\Theta^{\frac{\delta}{2}} x_0}{\dot{H}^0} \\
		&\le \norm{\Lambda^{-\frac{\alpha}{2}}}{\cL(\dot{H}^0)} \norm{P^1(\tilde{E}(t_n)-E(t_n))\Theta^{-\frac{\delta}{2}}}{\cL(\cH,\dot{H}^0)} \norm{\Theta^{\frac{\delta}{2}} x_0}{\cH} \lesssim h^{\min(\delta \frac{\kappa}{\kappa+1},\kappa)} + {\Delta t}^{\min(\delta \frac{\rho}{\rho+1},1)}.
		\end{align*}
		Before treating $\mathrm{II}$, we consider the last term. Lemma~\ref{lem:BDG} yields 
		\begin{align*}
		\mathrm{III}^2 &\lesssim \int^{t_n}_0 \norm{\Lambda^{-\frac{\alpha}{2}} P^1\left(\tilde{E}({t_n}-s)BG(\dfloor{s})-E({t_n}-s)BG(s)\right)}{\cL_2^0}^2 \dd s \\
		&= \sum_{j=0}^{n-1} \int_{t_j}^{t_{j+1}} \norm{\Lambda^{-\frac{\alpha}{2}} P^1\left(\tilde{E}({t_n}-s)BG(t_j)-E({t_n}-s)BG(s)\right)}{\cL_2^0}^2 \dd s. 
		\end{align*}
		
		For the last integrand, we make the split
		\begin{align*}
		&\norm{\Lambda^{-\frac{\alpha}{2}} P^1\left(\tilde{E}({t_n}-s)BG(t_j)-E({t_n}-s)BG(s)\right)}{\cL_2^0} \\ &\quad\le \norm{\Lambda^{-\frac{\alpha}{2}} P^1\left(\tilde{E}({t_n}-s)-E({t_n}-s)\right)BG(t_j)}{\cL_2^0}  \\
		&\quad\quad+ \norm{\Lambda^{-\frac{\alpha}{2}} P^1\left(E({t_n}-s)B(G(t_j)-G(s)\right)}{\cL_2^0} \\
		&= \mathrm{IV} + \mathrm{V}.
		\end{align*}
		For the first of these terms, by~\eqref{eq:schatten_bound_1},~\eqref{eq:B_bound_2} and Lemma~\ref{lem:semigroup_error}, and since $\Lambda^{-\frac{\alpha}{2}}$ is a bounded operator, we get
		\begin{align*}
		\mathrm{IV} &\le \norm{\Lambda^{-\frac{\alpha}{2}}}{\cL(\dot{H}^0)} \norm{P^1\left(\tilde{E}(t_n-s)-E(t_n-s)\right)B\Lambda^{\frac{1-\beta}{2}}}{\cL(\dot{H}^0)} \norm{\Lambda^{\frac{\beta-1}{2}}G(t_j)}{\cL_2^0} \\
		&\lesssim
		\norm{P^1\left(\tilde{E}(t_n-s)-E(t_n-s)\right)B\Lambda^{\frac{1-\beta}{2}}}{\cL(\dot{H}^0)} \\
		&=\norm{P^1\left(\tilde{E}(t_n-s)-E(t_n-s)\right)\Theta^{-\frac{\beta}{2}} B\Lambda^{\frac{1}{2}}}{\cL(\dot{H}^0)} \\
		&\lesssim \norm{P^1\left(\tilde{E}(t_n-s)-E(t_n-s)\right)\Theta^{-\frac{\beta}{2}} }{\cL(\dot{H}^0)} \le h^{\min(\beta \frac{\kappa}{\kappa+1},\kappa)} + {\Delta t}^{\min(\beta \frac{\rho}{\rho+1},1)}.
		\end{align*}
		Furthermore, Lemma~\ref{lem:semigroup_error} also implies that
		\begin{align*}
		&\norm{\Lambda^{-\frac{\alpha}{2}}P^1\left(\tilde{E}(t_n-s)-E(t_n-s)\right)B\Lambda^{\frac{1-\beta}{2}}}{\cL(\dot{H}^0)} \\
		&\quad=
		\norm{\Lambda^{-\frac{\alpha}{2}}P^1\left(\tilde{E}(t_n-s)-E(t_n-s)\right) B\Lambda^{\frac{1-\alpha}{2}} }{\cL(\dot{H}^0)} \norm{\Lambda^{\frac{\alpha-\beta}{2}}}{\cL(\dot{H}^0)}\\
		&\quad\lesssim  h^{2 \alpha \frac{\kappa}{\kappa+1}} + {\Delta t}^{\min(2 \alpha \frac{\rho}{\rho+1},1)},
		\end{align*}
		using also the fact that $\alpha \le r \le \beta$ and that $\alpha \le 2 \le \kappa$ in the last inequality. Therefore
		\begin{equation*}
		\mathrm{IV} \lesssim  h^{\min(\max(2\alpha,\beta) \frac{\kappa}{\kappa+1},\kappa)} + {\Delta t}^{\min(\max(2\alpha,\beta) \frac{\rho}{\rho+1},1)}.
		\end{equation*}
		Next, for $\mathrm{V}$, by~\eqref{eq:schatten_bound_1}, \eqref{eq:semigroup_bound},  Assumption~\ref{assumptions:1} and~\eqref{eq:B_bound_2}, we obtain
		\begin{align*}
		\mathrm{V} &= \norm{\Lambda^{-\frac{\alpha+\beta}{2}} P^1(E({t_n}-s)B\Lambda^{\frac{1}{2}}\Lambda^{\frac{\beta-1}{2}}(G(t_j)-G(s))}{\cL_2^0} \\
		&\le \norm{\Lambda^{-\frac{\alpha+\beta}{2}}}{\cL(\dot{H}^0)} \norm{E({t_n}-s)}{\cL(\cH)} \norm{B\Lambda^\frac{1}{2}}{\cL(\dot{H}^0,\cH)} 
		\norm{\Lambda^{\frac{\beta-1}{2}} \left(G(t_j)-G(s)\right)}{\cL_2^0} \lesssim |t_j-s|^\eta.
		\end{align*}
		As a consequence, we arrive at
		\begin{equation*}
		\mathrm{III} \lesssim \left(h^{\min(\max(2\alpha,\beta) \frac{\kappa}{\kappa+1},\kappa)} + {\Delta t}^{\min(\max(2\alpha,\beta) \frac{\rho}{\rho+1},\eta,1)}  \right).
		\end{equation*}
		We now continue with $\mathrm{II}$ and note that
		\begin{equation*}
		\mathrm{II} = \sum_{j=0}^{n-1} \int_{t_j}^{t_{j+1}} \norm{\Lambda^{-\frac{\alpha}{2}}P^1\left(\tilde{E}(t_n-s)BF(t_{j}, \tilde{X}_1(t_{j})) - E(t_n-s)BF(s, X_1(s))\right)}{L^p(\Omega,\dot{H}^0)} \dd s. 
		\end{equation*}
		We split the integrand with
		\begin{align*}
		&\norm{\Lambda^{-\frac{\alpha}{2}}P^1\left(\tilde{E}(t_n-s)BF(t_{j}, \tilde{X}_1(t_{j})) - E(t_n-s)BF(s, X_1(s))\right)}{L^p(\Omega,\dot{H}^0)} \\
		&\quad \le 
		\norm{\Lambda^{-\frac{\alpha}{2}}P^1\left(\tilde{E}(t_n-s)-E(t_n-s)\right)BF(t_{j}, \tilde{X}_1(t_{j}))}{L^p(\Omega,\dot{H}^0)}  \\
		&\quad\quad+ \norm{\Lambda^{-\frac{\alpha}{2}}P^1E(t_n-s)B\left(F(t_{j}, \tilde{X}_1(t_{j}))-F(t_{j}, {X}_1(t_{j})) \right)}{L^p(\Omega,\dot{H}^0)} 
		\\
		&\quad\quad+ \norm{\Lambda^{-\frac{\alpha}{2}}P^1E(t_n-s)B\left(F(t_{j}, {X}_1(t_{j}))-F(t_j, {X}_1(s)) \right)}{L^p(\Omega,\dot{H}^0)}  \\
		&\quad\quad+ \norm{\Lambda^{-\frac{\alpha}{2}}P^1E(t_n-s)B\left(F(t_{j}, {X}_1(s))-F(s, {X}_1(s)) \right)}{L^p(\Omega,\dot{H}^0)}  \\
		&\quad= \mathrm{VI} + \mathrm{VII} + \mathrm{VIII}+ \mathrm{IX}.	
		\end{align*}
		For the first of these terms, under Assumption~\ref{assumptions:1}, by Lemma~\ref{lem:semigroup_error} and~\eqref{eq:B_bound_2}, using the assumption that $\theta \le \min(\beta,\delta,1) \le \min(r,1)$, we get
		\begin{align*}
		\mathrm{VI} &\le
		\norm{\Lambda^{-\frac{\alpha}{2}}}{\cL(\dot{H}^0)}
		\norm{P^1\left(\tilde{E}(t_n-s)-E(t_n-s)\right)B\Lambda^{\frac{\theta}{2}}}{\cL(\dot{H}^0)} \norm{ \Lambda^{\frac{\theta}{2}}F(t_{j}, \tilde{X}_1(t_{j}))}{L^p(\Omega,\dot{H}^0)} \\
		&\lesssim \norm{P^1\left(\tilde{E}(t_n-s)-E(t_n-s)\right) \Theta^{-\frac{1+\theta}{2}}}{\cL(\cH,\dot{H}^0)} \norm{B\Lambda^{\frac{1}{2}} }{\cL(\dot{H}^0,\cH)}\left(1 + \norm{ \Lambda^{\frac{\theta}{2}}\tilde{X}_1(t_{j})}{L^p(\Omega,\dot{H}^0)} \right)\\
		&\lesssim  \left(h^{(1+\theta)\frac{\kappa}{\kappa+1}} + {\Delta t}^{\min((1+\theta)\frac{\rho}{\rho+1},1)}\right).
		\end{align*}
		Since $\alpha \le r \le 1+\theta \le 2$, similarly to term $\mathrm{IV}$, Lemma~\ref{lem:semigroup_error} also implies that
		\begin{align*}
		\mathrm{VI} &\lesssim
		\norm{\Lambda^{-\frac{\alpha}{2}}P^1\left(\tilde{E}(t_n-s)-E(t_n-s)\right)B\Lambda^{\frac{1-\alpha}{2}}}{\cL(\dot{H}^0)}
		\norm{ \Lambda^{\frac{\alpha-1}{2}} F(t_{j}, \tilde{X}_1(t_{j}))}{L^p(\Omega,\dot{H}^0)} \\
		&\lesssim 
		\norm{\Lambda^{-\frac{\alpha}{2}}P^1\left(\tilde{E}(t_n-s)-E(t_n-s)\right)B\Lambda^{\frac{1-\alpha}{2}}}{\cL(\dot{H}^0)}
		\norm{ \Lambda^{\frac{\theta}{2}} F(t_{j}, \tilde{X}_1(t_{j}))}{L^p(\Omega,\dot{H}^0)} \\
		&\lesssim  \left(h^{2\alpha\frac{\kappa}{\kappa+1}} + {\Delta t}^{\min(2\alpha\frac{\rho}{\rho+1},1)}\right)\left(1 + \norm{
			\Lambda^{\frac{\theta}{2}} \tilde{X}_1(t_{j})}{L^p(\Omega,\dot{H}^0)} \right),
		\end{align*}
		whence 
		\begin{equation*}
		\mathrm{VI} \lesssim \left(h^{\max(2\alpha,1+\theta) \frac{\kappa}{\kappa+1}} + {\Delta t}^{\min(\max(2\alpha,1+\theta) \frac{\rho}{\rho+1},1)}  \right).
		\end{equation*}
		Term $\mathrm{VII}$ can be bounded using~\eqref{eq:P^1EB_Theta_commute}, \eqref{eq:semigroup_bound} and~\eqref{eq:B_bound_2}, as
		\begin{align*}
		\mathrm{VII} &= \norm{P^1E(t_n-s)B\Lambda^\frac{1}{2} \Lambda^{-\frac{1+\alpha}{2}}\left(F(t_{j}, \tilde{X}_1(t_{j}))-F(t_{j}, {X}_1(t_{j})) \right)}{L^p(\Omega,\dot{H}^0)} \\
		&\lesssim \norm{\Lambda^{-\frac{1+\alpha}{2}}\left(F(t_{j}, \tilde{X}_1(t_{j}))-F(t_{j}, {X}_1(t_{j})) \right)}{L^p(\Omega,\dot{H}^0)}.
		\end{align*}
		If only Assumption~\ref{assumptions:1} holds, we directly obtain
		\begin{equation*}
		\mathrm{VII} \lesssim \norm{ \tilde{X}_1(t_{j})-{X}_1(t_{j}) }{L^p(\Omega,\dot{H}^{0})}.
		\end{equation*}
		If also Assumption~\ref{assumptions:2} holds with $\alpha = \nu \ge \mu -1 $, then, since $\cG^1_\mathrm{p}(\dot{H}^0,\dot{H}^{-(1+\nu)}) \subset \cG^1_\mathrm{p}(\dot{H}^0,\dot{H}^{-\mu})$, we may use the mean value theorem along with~\eqref{eq:ass2:munu} to deduce that
		\begin{align*}
		\mathrm{VII} 
		& \le \int^1_0 \norm{\Lambda^{-\frac{1+\nu}{2}} F'\left(t_j,s\tilde{X}_1(t_{j})+(1-s){X}_1(t_{j})\right)\left(\tilde{X}_1(t_{j})-{X}_1(t_{j})\right)}{L^p(\Omega,\dot{H}^0)} \dd s \\ & \lesssim \int^1_0 \norm{\Lambda^{-\frac{\mu}{2}} F'\left(t_j,s\tilde{X}_1(t_{j})+(1-s){X}_1(t_{j})\right)\left(\tilde{X}_1(t_{j})-{X}_1(t_{j})\right)}{L^p(\Omega,\dot{H}^0)} \dd s \\
		&\lesssim \left(1 +
		\norm{\Lambda^{\frac{\nu}{2}}\tilde{X}_1(t_{j})}{L^p(\Omega,\dot{H}^0)} +
		\norm{\Lambda^{\frac{\nu}{2}}{X}_1(t_{j})}{L^p(\Omega,\dot{H}^0)}\right) 
		\norm{ \Lambda^{-\frac{\nu}{2}} \left( \tilde{X}_1(t_{j})-{X}_1(t_{j})\right)}{L^p(\Omega,\dot{H}^0)}
		\\
		&\lesssim \left(1 + \norm{\Lambda^{\frac{\min(r,1)}{2}}\tilde{X}_1(t_{j})}{L^p(\Omega,\dot{H}^0)}+
		\norm{\Lambda^{\frac{\min(r,1)}{2}}{X}_1(t_{j})}{L^p(\Omega,\dot{H}^0)}\right) \\
		&\quad\quad\times\norm{ \Lambda^{-\frac{\nu}{2}} \left( \tilde{X}_1(t_{j})-{X}_1(t_{j})\right)}{L^p(\Omega,\dot{H}^0)} \lesssim \norm{ \Lambda^{-\frac{\nu}{2}} \left( \tilde{X}_1(t_{j})-{X}_1(t_{j})\right)}{L^p(\Omega,\dot{H}^0)},
		\end{align*}
		where we also applied Theorem~\ref{thm:X_spat_reg} and~\eqref{eq:fully_discrete_spatial_regularity} in the last step.
		Similarly, we have
		\begin{equation*}
		\mathrm{VIII} \lesssim \norm{ X_1(t_{j})-X_1(s) }{L^p(\Omega,\dot{H}^{0})} \lesssim |t_j - s|^{\min({r,1})}
		\end{equation*}
		if we only consider Assumption~\ref{assumptions:1}, while if also Assumption~\ref{assumptions:2} holds, then
		\begin{equation*}
		\mathrm{VIII} \lesssim 
		\norm{ \Lambda^{-\frac{\nu}{2}} \left( X_1(t_{j})-{X}_1(s)\right)}{L^p(\Omega,\dot{H}^0)} \lesssim |t_j - s|^{\min(2\nu,1)},
		\end{equation*}
		using Theorem~\ref{thm:X_spat_reg} and Theorem~\ref{thm:mild_holder} in the last step.
		For the last term, by~\eqref{eq:P^1EB_Theta_commute},~\eqref{eq:semigroup_bound} and~\eqref{eq:B_bound_2},
		\begin{align*}
		\mathrm{IX} &\lesssim \norm{\Lambda^{-\frac{1 }{2}}\left(F(t_{j}, {X}_1(s))-F(s, {X}_1(s)) \right)}{L^p(\Omega,\dot{H}^0)} \lesssim \left(1 + \norm{X_1(s)}{L^p(\Omega,\dot{H}^0)}\right) |t_j-s|^\eta \\ &\lesssim |t_j-s|^\eta,
		\end{align*}
		where we have applied Theorem~\ref{thm:X_spat_reg} and Assumption~\ref{assumptions:1}\ref{assumptions:1:F} in the last inequality. 
		
		Collecting the bounds 
		on terms $\mathrm{VI}, \mathrm{VII}, \mathrm{VIII}$ and $\mathrm{IX}$, we have shown that under Assumption~\ref{assumptions:1}, with $\alpha = 0$,
		\begin{align*}
		\mathrm{II} \lesssim h^{(1+\theta) \frac{\kappa}{\kappa+1}} + {\Delta t}^{\min((1+\theta) \frac{\rho}{\rho+1},r,1)} + \Delta t\sum_{j=0}^{n-1} \norm{ \tilde{X}_1(t_{j})-{X}_1(t_{j}) }{L^p(\Omega,\dot{H}^{0})}.
		\end{align*}
		Taking also into account the bound on term $\mathrm{I}$ and that
		\begin{equation*}
		\mathrm{III} \lesssim \left(h^{\min(\beta \frac{\kappa}{\kappa+1},\kappa)} + {\Delta t}^{\min(\beta \frac{\rho}{\rho+1},\eta,1)}  \right),
		\end{equation*}
		we obtain~\eqref{eq:fully_discrete_strong_converence} by Lemma~\ref{lem:Gronwall}.
		
		On the other hand, if Assumption~\ref{assumptions:2} also holds, the bound on term $\mathrm{I}$ remains the same while with $\alpha = \nu$,
		\begin{align*}
		\mathrm{II} &\lesssim h^{\min(\max(2\nu,1+\theta) \frac{\kappa}{\kappa+1},\kappa)} + {\Delta t}^{\min(\max(2\nu,1+\theta) \frac{\rho}{\rho+1},\max(2\nu,r),\eta,1)}  \\
		&\quad\quad\quad+ \Delta t \sum_{j=0}^{n-1}  \norm{ \Lambda^{-\frac{\nu}{2}} (\tilde{X}_1(t_{j})-{X}_1(t_{j})) }{L^p(\Omega,\dot{H}^{0})}. 
		\end{align*}
		Taking also into account that
		\begin{equation*}
		\mathrm{III} \lesssim \left(h^{\min(\max(2\nu,\beta) \frac{\kappa}{\kappa+1},\kappa)} + {\Delta t}^{\min(\max(2\nu,\beta) \frac{\rho}{\rho+1},\eta,1)}  \right),
		\end{equation*} by another application of Lemma~\ref{lem:Gronwall}, we obtain
		\begin{equation}
		\label{eq:Delta_t_exponent}
		\begin{split}
		&\sup_{n \in \{0,1,\ldots,N_{\Delta t}\}} \norm{\tilde{X_1}(t_n)-X_1(t_n)}{L^p(\Omega,\dot{H}^{-\nu})} \\ 
		&\hspace{70pt}\lesssim h^{r' \frac{\kappa}{\kappa+1}} + {\Delta t}^{\min(\max(2\nu,\beta) \frac{\rho}{\rho+1},\max(2\nu,1+\theta) \frac{\rho}{\rho+1},\delta \frac{\rho}{\rho+1},\max(2\nu,r),\eta,1)}.
		\end{split}
		\end{equation}
		We now note that 
		\begin{equation*} \min(\max(2\nu,\beta),\max(2\nu,1+\theta),\delta) \le \max(2\nu,r),
		\end{equation*}
		so, since ${\rho}/{(\rho+1)} < 1$, it holds that either $\max(2\nu,r) > \max(2\nu,\beta){\rho}/{(\rho+1)}$ or $\max(2\nu,r) > \max(2\nu,1+\theta){\rho}/{(\rho+1)}$ or $\max(2\nu,r) > \delta{\rho}/{(\rho+1)}$. No matter which of these three cases occur, the minimum in the exponent of $\Delta t$ in~\eqref{eq:Delta_t_exponent} is not attained at $\max(2\nu,r)$.
		In other words, 
		\begin{align*}
		&{\min(\max(2\nu,\beta) \frac{\rho}{\rho+1},\max(2\nu,1+\theta) \frac{\rho}{\rho+1},\delta \frac{\rho}{\rho+1},\max(2\nu,r),\eta,1)} \\
		&\quad= {\min(\max(2\nu,\beta) \frac{\rho}{\rho+1},\max(2\nu,1+\theta) \frac{\rho}{\rho+1},\delta \frac{\rho}{\rho+1},\eta,1)} = \min(r' \frac{\rho}{\rho+1},\eta,1).
		\end{align*}
		With this, the proof is completed.
	\end{proof}
	
	\subsection{Weak convergence and Malliavin calculus}
	
	To deduce a result on the weak convergence of the approximation~\eqref{eq:fully_discrete_mild_solution} to the mild solution given by~\eqref{eq:mild_solution}, we use Malliavin calculus for which we briefly review some definitions and results from~\cite{AKL16} (the authors therein consider so called \textit{refined} Sobolev--Malliavin spaces whereas we only need classical ones, hence the difference in notation below).
	To avoid technicalities we assume from here on that the filtration $(\cF_t)_{t \in [0,T]}$ is generated by the Wiener process $W$. Let $I(\varphi) = \int^T_0 \varphi(t) \dd W(t)$ for $\varphi \in L^2([0,T], \cL_2(H_0,\R)) \simeq L^2([0,T], H_0)$ and let $\cS $ denote the set of all cylindrical random variables of the form $F = f(I(\varphi_1), \ldots, I(\varphi_N))$ for $f \in \cC^1_{\mathrm{p}}(\R^N,\R)$ and a sequence $(\phi_j)_{j=1}^N \subset L^2([0,T], H_0)$, $N \in \N$. 
	The definition of the Malliavin derivative of $F \in \cS$ is given by $D_{\cdot}F = \sum^N_{j=1} \partial_j f(I(\varphi_1), \ldots, I(\varphi_N)) \varphi_j(\cdot) \in L^2([0,T]\times \Omega, H_0)$. For a given real separable Hilbert space $U$ we let $\cS(U)$ be the space of all $U$-valued random variables of the form $Y = \sum_{j=1}^{M} F_j v_j$ with $F_j \in \cS$ and $v_j \in U$ for $j=1, 2, \ldots, M$, and define $D Y = \sum_{j=1}^{M} v_j \otimes D F_j  \in U \otimes L^2([0,T]\times \Omega, H_0)$. Note that $U \otimes L^2([0,T]\times \Omega, H_0) \simeq L^2([0,T]\times \Omega, \cL_2(H_0,U)) \simeq L^2(\Omega, L^2([0,T], \cL_2(H_0,U)))$. This definition of $D$ does not depend on the specific representation of $Y$. The operator $D\colon \cS(U) \subset L^p(\Omega,U) \to L^p(\Omega,L^2([0,T], \cL_2(H_0,U))$ is closable for any $p>1$ and we write $\D^{1,p}(U)$ for the closure of $\cS(U)$ in $L^p(\Omega,U)$ with respect to the norm
	\begin{equation*}
	\norm{F}{\D^{1,p}(U)} = \left(\E[\norm{F}{U}^p]+\E[\norm{DF}{L^2([0,T],\cL_2(H_0,U))}^p]\right)^\frac{1}{p}.
	\end{equation*}
	Next, we recall some of the basic properties of the Malliavin derivative. 
	First of all, any deterministic element is Malliavin differentiable with Malliavin derivative zero. For predictable processes $\Phi \in L^2([0,T]\times \Omega, \cL_2(H_0,U))$ and any $F \in \D^{1,2}(U)$, the following equality holds, sometimes referred to as Malliavin integration by parts,
	\begin{equation}
	\label{eq:malliavin_integration_by_parts}
	\inpro[L^2(\Omega,U)]{F}{\int^T_0 \Phi(t) \dd W(t)} = \int_{0}^{T} \inpro[L^2(\Omega, \cL_2(H_0,U))]{D_r F}{\Phi(r)} \dd r.
	\end{equation}
	We also need to know how the Malliavin derivative acts on stochastic integrals, but restrict ourselves to the case that the integrand $\Phi \in L^2([0,T], \cL_2(H_0,U))$ is deterministic. Then, for all $t \in [0,T]$, it follows that $\int^t_0 \Phi(r) \dd W(r) \in \D^{1,p}(U)$ for all $p > 1$ and
	\begin{equation}
	\label{eq:malliavin_on_ito_integral}
	D \int^t_0 \Phi(r) \dd W(r) = \chi_{(0,t]}(\cdot) \Phi(\cdot).
	\end{equation}
	For Lebesgue integrals of stochastic processes on the other hand, Malliavin differentiation and integration simply commute (see~\cite[Proposition~4.8]{K14}). If $\Phi\colon [0,T] \times \Omega \to U$ fulfill $\Phi \in \D^{1,p}(L^2([0,T],U))$ for some $p \ge 2$, then $\int^T_0 \Phi(t) \dd t \in \D^{1,p}(U)$ and 
	\begin{equation}
	\label{eq:malliavin_on_det_integral}
	D \int^T_0 \Phi(t) \dd t = \int^T_0 D \Phi(t) \dd t.
	\end{equation}
	It is known (cf. \cite[Section~1.3]{N06}) that a sufficient condition for $\Phi \in \D^{1,p}(L^2([0,T],U))$ is that $\Phi(t) \in \D^{1,p}(U)$ for all $t \in [0,T]$ and that $\sup_{t \in [0,T]} \norm{D \Phi(t)}{L^2(\Omega\times[0,T],\cL_2(H_0,U))} < \infty$. It is also worth mentioning that $D$ commutes with any bounded linear operator between Hilbert spaces. For nonlinear mappings $\varphi \in \cG_\mathrm{p}^1(U,V)$, where $V$ is another arbitrary real separable Hilbert space, a chain rule holds instead. If there is a $q\ge 1$ and a constant $C>0$ such that $\norm{\varphi(x)}{V} \le C ( 1 + \norm{x}{U}^{q})$ and $\norm{\varphi'(x)}{\cL(U,V)} \le C (1 + \norm{x}{U}^{q-1})$ for all $x \in U$, then for all $p > 1$ and $F \in \D^{1,pq}(U)$, it holds that $\varphi(F) \in \D^{1,p}(V)$ and 
	\begin{equation}
	\label{eq:malliavin_chain_rule}
	D\varphi(F) = \varphi'(F)DF.
	\end{equation} 
	
	With these results in place, to be able to deduce a weak convergence result, we need to impose a stronger condition on $G$  in~\eqref{eq:ito_wave_equation} and on the covariance operator $Q$ of the Wiener process. We also take this opportunity to specify our assumptions on the test function $\phi\colon \dot{H}^0 \to \R$. We identify $\cL(\dot{H}^0,\R)$ with $\dot{H}^0$ so that, for every $u \in \dot{H}^0$, $\phi'(u) \in \dot{H}^0$.  
	
	\begin{assumption}
		\label{assumptions:3}
		The following conditions hold:
		\hfill
		\begin{enumerate}[label=(\roman*)]
			\item \label{ass:3:phi} $\phi \in \cG^2_{\mathrm{p}}(\dot{H}^0,\R)$,
			\item \label{ass:3:G} $G(t)=I$ for all $t \in [0,T]$, and either
			\item \label{ass:3:origQ} $\norm{\Lambda^{\beta-\frac{1}{2}}Q \Lambda^{-\frac{1}{2}}}{\trace} < \infty$ or 
			\item \label{ass:3:altQ} $\trace(Q) < \infty$ and $\Lambda^{-1/2} \phi''(u) =  \phi''(u) \Lambda^{-1/2}$ for all $u \in \dot{H}^0$.
		\end{enumerate} 
	\end{assumption}
	
	The condition Assumption~\ref{assumptions:3}\ref{ass:3:origQ}, $\norm{\Lambda^{\beta-\frac{1}{2}}Q \Lambda^{-\frac{1}{2}}}{\trace} < \infty$, implies the corresponding condition $\norm{\Lambda^{\frac{\beta-1}{2}}}{\cL_2^0} < \infty$ of Assumption~\ref{assumptions:1}\ref{assumptions:1:G}. They are equivalent in the important case that $Q=I$ or more generally when $Q\Lambda = \Lambda Q$, see~\cite[Theorem~2.1]{KLL12}. We also have the following simple but useful lemma.
	\begin{lemma}
		\label{lem:Malliavin_1}
		Suppose that $\norm{\Lambda^{\beta-\frac{1}{2}}Q \Lambda^{-\frac{1}{2}}}{\trace} < \infty$. Then $\cL(\dot{H}^{-1},\dot{H}^{0}) \subset \cL_2^0$ and for $\Gamma \in \cL(\dot{H}^{-1},\dot{H}^{0})$ there exists a constant $C> 0$ such that $$\norm{\Gamma}{\cL_2^0} \le C \norm{\Gamma}{\cL(\dot{H}^{-1},\dot{H}^{0})} = C \norm{\Gamma \Lambda^{\frac{1}{2}}}{\cL(\dot{H}^{0})}.$$
	\end{lemma}
	\begin{proof}
		The claim is a consequence of~\eqref{eq:HSvTrace},~\eqref{eq:trace_cyclic} and~\eqref{eq:schatten_bound_1} via the calculation
		\begin{align*}
		\norm{\Gamma}{\cL_2^0}^2 
		&= \left|\trace\left(\Gamma Q  \Gamma^* \right)\right|
		=  \left|\trace\left(\Gamma^* \Gamma Q \right)\right|
		= \left|\trace\left(\left(\Gamma\Lambda^\frac{1}{2}\right)^* \Gamma \Lambda^{\frac{1}{2}} \Lambda^{- \beta} \Lambda^{\beta - \frac{1}{2}} Q \Lambda^{- \frac{1}{2}} \right)\right| \\
		&\le \norm{\Gamma \Lambda^{\frac{1}{2}}}{\cL(\dot{H}^0)}^2
		\norm{\Lambda^{-\beta}}{\cL(\dot{H}^0)} \norm{\Lambda^{\beta-\frac{1}{2}}Q \Lambda^{-\frac{1}{2}}}{\trace} \\
		&= \norm{\Gamma }{\cL(\dot{H}^{-1},\dot{H}^{0})}^2
		\norm{\Lambda^{-\beta}}{\cL(\dot{H}^0)} \norm{\Lambda^{\beta-\frac{1}{2}}Q \Lambda^{-\frac{1}{2}}}{\trace}.\qedhere
		\end{align*}
	\end{proof}
	
	The reason for the alternative Assumption~\ref{assumptions:3}\ref{ass:3:altQ} to Assumption~\ref{assumptions:3}\ref{ass:3:origQ} is that  the condition $\norm{\Lambda^{\beta-\frac{1}{2}}Q \Lambda^{-\frac{1}{2}}}{\trace} < \infty$ is hard to interpret in the common (cf. \cite[Corollary~4.9]{B05}, where $Q$ is trace-class and induced by a covariance kernel) case that $\Lambda Q \neq Q \Lambda$. Note that Assumption~\ref{assumptions:3}\ref{ass:3:altQ} implies that $\norm{\Lambda^{\frac{\beta-1}{2}}}{\cL_2^0} < \infty$ for all $\beta \le 1$.
	
	As a first step towards our weak convergence result, we need the following regularity estimate for the Malliavin derivatives of the first component of the mild solution given by~\eqref{eq:mild_solution} and its approximation.
	\begin{lemma}
		\label{lem:Malliavin_2}
		Let Assumptions~\ref{assumptions:1} and~\ref{assumptions:2} hold. Let $X$ be the mild solution given by~\eqref{eq:mild_solution} of the stochastic wave equation and let $\tilde{X}$ be the fully discrete approximation given by~\eqref{eq:fully_discrete_mild_solution}. Under Assumption~\ref{assumptions:3}\ref{ass:3:G}-\ref{ass:3:origQ},  
		$X_1(t) \in \D^{1,p}(\dot{H}^0)$ for all $p \ge 2, t \in [0,T]$, and
		\begin{equation}
		\label{eq:malliavin_regularity_mild}
		\sup_{t \in [0,T]} \norm{D X_1(t) \Lambda^{\frac{1}{2}}}{L^2([0,T]\times\Omega,\cL(\dot{H}^0))} < \infty.
		\end{equation}
		Furthermore, $\tilde X_1(t_n) \in \D^{1,p}(\dot{H}^0)$ for all $p \ge 2$, $n \in \{0,1,\ldots,N_{\Delta t}\}$, and
		\begin{equation}
		\label{eq:malliavin_regularity_discrete}
		\sup_{\substack{n \in \{0,1,\ldots,N_{\Delta t}\} \\ h, \Delta t \in (0,1]}} \norm{D \tilde X_1(t_n) \Lambda^{\frac{1}{2}}}{L^2([0,T]\times\Omega,\cL(\dot{H}^0))} < \infty.
		\end{equation}
		If Assumption~\ref{assumptions:3}\ref{ass:3:altQ} holds in place of Assumption~\ref{assumptions:3}\ref{ass:3:origQ}, similar statements hold with~\eqref{eq:malliavin_regularity_mild} replaced by
		\begin{equation}
		\label{eq:malliavin_regularity_mild_2}
		\sup_{t \in [0,T]} \norm{\Lambda^{\frac{1}{2}} D X_1(t) }{L^2([0,T]\times\Omega,\cL(\dot{H}^0))} < \infty,
		\end{equation}
		and~\eqref{eq:malliavin_regularity_discrete} by
		\begin{equation}
		\label{eq:malliavin_regularity_discrete_2}
		\sup_{\substack{n \in \{0,1,\ldots,N_{\Delta t}\} \\ h, \Delta t \in (0,1]}} \norm{\Lambda^{\frac{1}{2}} D \tilde X_1(t_n) }{L^2([0,T]\times\Omega,\cL(\dot{H}^0))} < \infty.
		\end{equation}
	\end{lemma}
	\begin{proof}
		We start by showing~\eqref{eq:malliavin_regularity_mild}. Define the sequence $(X^n_1)_{n=0}^\infty$ by $X^0_1 = 0$ and, for $t \in [0,T]$ and $n \ge 0$,  
		\begin{equation}
		\label{eq:Xn-sequence}
		X^{n+1}_1(t) = P^1 E(t) x_0 + \int^t_0 P^1 E(t-s) B F(s,X^n_1(s)) \dd s + \int^t_0 P^1 E(t-s)B \dd W(s).
		\end{equation}
		Since the existence result~\cite[Theorem~7.2]{DPZ14} that we cited in Theorem~\ref{thm:X_spat_reg} is proven via a fixed point argument for this sequence, it follows that $\lim_n X_1^{n}(t) = X_1(t)$ in $L^p(\Omega,\dot{H}^0)$ for all $t \in [0,T]$ and $p \ge 2$. By~\eqref{eq:B_bound_2}, \eqref{eq:P^1EB_Theta_commute} and~\eqref{eq:semigroup_bound}, we have \begin{equation}
		\label{eq:mallPEBbound}
		\sup_{t \in [0,T]}
		\norm{P_1E(t)B}{L(\dot{H}^{\alpha-1},\dot{H}^{\alpha})} < \infty.
		\end{equation} This implies, along with Assumption~\ref{assumptions:2}, that $P^1 E(t-s)BF(t, \cdot) \in \cG^{1}_{\mathrm{b}}(\dot{H}^0,\dot{H}^{1-\min(\mu,1)}) \subset \cG^{1}_{\mathrm{b}}(\dot{H}^0,\dot{H}^0)$ for all $s,t \in [0,T]$. The chain rule for the Malliavin derivative is then applicable so that $P^1 E(t-s)BF(t,X^n_1(t)) \in \D^{1,2}(\dot{H}^0)$ for all $s,t \in [0,T]$ as long as $X_1^n(t) \in \D^{1,2}(\dot{H}^0)$ for all $t \in [0,T]$, and that in this case $DP^1 E(t-s)BF(t,X^n_1(t)) = P^1 E(t-s)BF'(t,X^n_1(t)) DX^n_1(t)$ for all $s,t \in [0,T]$. Therefore, we may apply $D$ to both sides of~\eqref{eq:Xn-sequence} and hence, by the fact that the Malliavin derivative of a deterministic element is zero,~\eqref{eq:malliavin_on_ito_integral} and~\eqref{eq:malliavin_on_det_integral}, we get 
		\begin{equation*}
		D X^{n+1}_1(t) = \int^t_0 P^1 E(t-r) B F'(r,X^n_1(r)) D X^n_1(r) \dd r + \chi_{(0,t]}(\cdot) P^1 E(t-\cdot) B
		\end{equation*}
		for all $t \in [0,T]$. Our aim is now to show that the sequence $(DX^n_1)_{n=0}^\infty$ has a limit in $\cC([0,T],L^p(\Omega\times[0,T],\cL(\dot{H}^{-1},\dot{H}^0)))$ for any $p \ge 2$. To this end, note first that the mapping $[0,T] \ni t \mapsto \chi_{(0,t]}(\cdot) P^1 E(t-\cdot) B$ is in this space by~\eqref{eq:mallPEBbound} with $\alpha= 0$. Next, we show that there is an equivalent norm $\norm{\cdot}{\sigma}$ on $\cC([0,T],L^p(\Omega\times[0,T],\cL(\dot{H}^{-1},\dot{H}^0)))$ such that the mapping
		\begin{equation*}
		[0,T] \ni t \mapsto I(X,Y)_t = \int^t_0 P^1 E(t-r) B F'(r,X(r)) Y(r) \dd r,
		\end{equation*}
		where $X \in L^2(\Omega\times[0,T],\dot{H}^0)$ and $Y \in \cC([0,T],L^p(\Omega\times[0,T],\cL(\dot{H}^{-1},\dot{H}^0)))$, fulfills 
		\begin{equation}
		\label{eq:I_contraction}
		\norm{I(X,Y)}{\sigma} \le \alpha \norm{Y}{\sigma}
		\end{equation}
		for some $\alpha \in [0,1)$ and all $X \in L^2(\Omega\times[0,T],\dot{H}^0)$. We choose, for $\sigma \ge 0$ to be determined, $\norm{Y}{\sigma} = \sup_{t \in [0,T]} e^{-\sigma t} \norm{Y(t)}{L^p(\Omega\times[0,T],\cL(\dot{H}^{-1},\dot{H}^0))}$ and note that for $t \in [0,T]$, by~\eqref{eq:semigroup_bound} and~\eqref{eq:ass2:derivativebounded},
		\begin{align*}
		&\norm{I(X,Y)_t}{L^p(\Omega\times[0,T],\cL(\dot{H}^{-1},\dot{H}^0))} \\
		&\quad\le \int_{0}^{t} \norm{P^1 E(t-r)}{\cL(\cH,\dot{H}^0)} \norm{B F'(r,X(r))Y(r)}{L^p(\Omega\times[0,T],\cL(\dot{H}^{-1},\cH))} \dd r \\
		&\quad \lesssim \int^t_0 \norm{Y(r)}{L^p(\Omega\times[0,T],\cL(\dot{H}^{-1},\dot{H}^0))} \dd r \le \norm{Y}{\sigma} \int^t_0 e^{\sigma r} \dd r = \frac{e^{\sigma t}-1}{\sigma} \norm{Y}{\sigma}.
		\end{align*}
		This implies that $\norm{I(X,Y)}{\sigma} \lesssim \sigma^{-1} \norm{Y}{\sigma}$ so that~\eqref{eq:I_contraction} is fulfilled for sufficiently large $\sigma$. By the Banach fixed point theorem, therefore, $(DX^n_1)_{n=0}^\infty$ has a limit $\hat Y$ in $\cC([0,T],L^p(\Omega\times[0,T],\cL(\dot{H}^{-1},\dot{H}^0)))$. In particular, $\lim_n DX^n_1(t) = \hat{Y}(t)$ in $L^p(\Omega\times[0,T],\cL(\dot{H}^{-1},\dot{H}^0))$ for all $t \in [0,T]$. Thus, by Lemma~\ref{lem:Malliavin_1}, $\lim_n DX^n_1(t) = \hat{Y}(t)$ in $L^p(\Omega\times[0,T],\cL_2^0)$ for all $t \in [0,T]$. Since $D$ is closed and since $\lim_n X_1^{n}(t) = X_1(t)$ in $L^p(\Omega,\dot{H}^0)$ for all $t \in [0,T]$ this implies that $X_1(t) \in \D^{1,p}(\dot{H}^0)$ for all $t \in [0,T]$, and that $\hat{Y}(t) = D X_1(t)$ for all $t \in [0,T]$. With this, we have deduced~\eqref{eq:malliavin_regularity_mild}.
		
		Next, we move on to the corresponding results for the approximation. Since $P^1 X^0_{h,\Delta t} \in \D^{1,p}(\dot{H}^0)$ for all $p \ge 2$, a proof by induction using~\eqref{eq:fully_discrete_scheme}, \eqref{eq:malliavin_on_ito_integral} and~\eqref{eq:malliavin_on_det_integral} shows that $P^1 X^{n}_{h,\Delta t} = \tilde X_1(t_n) \in \D^{1,p}(\dot{H}^0)$ for all $p \ge 2, n \in \{0,1,\ldots,N_{\Delta t}\}$ and we have
		\begin{equation*}
		D P^1 X^{n}_{h,\Delta t} = \Delta t \sum^{n-1}_{j=0} P^1 E^{n-j}_{h,\Delta t} P_h B F(t_j,P^1 X^{j}_{h,\Delta t}) D P^1 X^{j}_{h,\Delta t} + \sum^{n-1}_{j=0} \chi_{(t_j,t_{j+1}]}(\cdot) P^1 E^{n-j}_{h,\Delta t} P_h B.
		\end{equation*}
		Therefore,~\eqref{eq:malliavin_regularity_discrete} follows by Lemma~\ref{lem:Gronwall} and we omit the details. 
		
		The results~\eqref{eq:malliavin_regularity_mild_2} and~\eqref{eq:malliavin_regularity_discrete_2} are deduced in the same way, by noting that if $\trace(Q) < \infty$, then $\cL(\dot{H}^0, \dot{H}^1) \subset \cL(\dot{H}^0) \subset \cL^0_2$ (cf. \cite[Lemma~2.3.7]{PR07}) and by using $\alpha=1$ instead of $\alpha = 0$ in~\eqref{eq:mallPEBbound}.
	\end{proof}
	
	The following error representation (cf. \cite[Theorem~5.9]{K14}) is a direct consequence of the mean value theorem,~\eqref{eq:malliavin_integration_by_parts}, \eqref{eq:malliavin_chain_rule}, and the facts that $D_\cG\phi(P^1 \cdot) = (P^1)^* \phi'(P^1 \cdot)$ and $D^2_\cG\phi(P^1 \cdot) = (P^1)^* \phi''(P^1 \cdot) P^1$, where $D_{\mathrm{\cG}}$ denotes the G\^ateaux derivative and $(P^1)^*$ the adjoint of $P^1$.
	
	\begin{proposition}
		\label{prop:error_rep}
		Let Assumptions~\ref{assumptions:1},~\ref{assumptions:2} and~\ref{assumptions:3} hold. Let $X$ be the mild solution given by~\eqref{eq:mild_solution} of the stochastic wave equation and let $\tilde{X}$ be the fully discrete approximation given by~\eqref{eq:fully_discrete_mild_solution}. Then, the weak error of the approximation satisfies
		\begin{align*}
		&\E\left[\phi(\tilde X_1(T))-\phi(X_1(T))\right] \\
		&\quad= \int^1_0 \E\Big[\Big\langle \phi'\left(X_1(T)+s(\tilde X_1 (T)-X_1(T))\right),
		P^1\left(\tilde{E}(T)-E(T)\right)x_0 \\
		&\quad\quad+ \int^T_0 P^1\left(\tilde{E}(T-r)BF(\dfloor{r}, \tilde{X}_1(\dfloor{r})) - E(T-r)BF(r, X_1(r))\right) \dd r
		\Big\rangle_{\dot{H}^0} \Big] \dd s \\
		&\quad\quad +
		\int^1_0 \int^T_0 \E\Big[\Big\langle \phi''\left(X_1(T)+s(\tilde X_1(T)-X_1(T))\right) \left( (1-s) D_r X_1(T) + sD_r \tilde X_1(T) \right), \\
		&\quad\hspace{30mm} P^1\left(\tilde{E}(T-r)-E(T-r)\right)B \Big\rangle_{\cL_2^0} \dd r \dd s.
		\end{align*}
	\end{proposition}
	
	We are now equipped to show a weak convergence result.
	
	\begin{theorem}[Weak convergence]
		\label{thm:weak_convergence}
		Let $X$ be the mild solution given by~\eqref{eq:mild_solution} of the stochastic wave equation and let $\tilde{X}$ be the fully discrete approximation given by~\eqref{eq:fully_discrete_mild_solution}. 
		Suppose that Assumptions~\ref{assumptions:1}, \ref{assumptions:FEM}, \ref{assumptions:2} and~\ref{assumptions:3} all hold and let $r' = \min(\max(2\nu,\beta),1+\theta,\delta)$. Then, for $\mu \le 1$, there exists a constant $C>0$ such that, for all $h, \Delta t \in (0,1]$,
		\begin{equation*}
		\left|\E\left[\phi(X_1(T))-\phi(\tilde X_1(T))\right]\right| \le C \big(h^{r' \frac{\kappa}{\kappa+1}} + {\Delta t}^{\min(r' \frac{\rho}{\rho+1},\eta,1)} \big).
		\end{equation*}
		If, on the other hand, $1 < \mu \le 2$, then there exists a constant $C>0$ such that, for all $h, \Delta t \in (0,1]$,
		\begin{equation*}
		\left|\E\left[\phi(X_1(T))-\phi(\tilde X_1(T))\right]\right| \le C \big(h^{r' \frac{\kappa}{\kappa+1}+1-\mu} + h^{1-\mu}{\Delta t}^{\min(r' \frac{\rho}{\rho+1},\eta,1)} \big).
		\end{equation*}
	\end{theorem}
	\begin{proof}
		We first prove the theorem under Assumption~\ref{assumptions:3}\ref{ass:3:origQ}.  Writing \begin{equation*}
		\hat{\phi}_1(s) = \phi'\left(X_1(T)+s(\tilde X_1 (T)-X_1(T))\right) 
		\end{equation*}
		and 
		\begin{equation*}
		\hat{\phi}_2(s) = \phi''\left(X_1(T)+s(\tilde X_1(T)-X_1(T))\right),
		\end{equation*}
		we use Proposition~\ref{prop:error_rep} to split the weak error
		\begin{align*}
		&\left|\E\left[\phi(\tilde X_1(T))-\phi(X_1(T))\right]\right| \\
		&\quad\le \int^1_0 \E\left[\left|\left\langle \hat{\phi}_1(s),
		P^1\left(\tilde{E}(T)-E(T)\right)x_0 \right\rangle_{\dot{H}^0} \right| \right] \dd s \\
		&\quad\quad+ \int^1_0 \E\Big[\Big|\Big\langle \hat{\phi}_1(s), \int^T_0 P^1\tilde{E}(T-r)BF(\dfloor{r}, \tilde{X}_1(\dfloor{r})) \\
		&\quad\hspace{23mm}- P^1 E(T-r)BF(r, X_1(r)) \dd r
		\Big\rangle_{\dot{H}^0} \Big| \Big] \dd s \\
		&\quad\quad+ \int^1_0 \int^T_0 \E\Big[\Big|\Big\langle \hat{\phi}_2(s)(1-s) D_r X_1(T),  P^1\left(\tilde{E}(T-r)-E(T-r)\right)B \Big\rangle_{\cL_2^0} \Big| \Big] \dd r \dd s \\
		&\quad\quad+ \int^1_0 \int^T_0 \E\Big[\Big|\Big\langle \hat{\phi}_2(s)s D_r \tilde X_1(T),  P^1\left(\tilde{E}(T-r)-E(T-r)\right)B \Big\rangle_{\cL_2^0} \Big| \Big] \dd r \dd s \\
		&\quad= \mathrm{I} + \mathrm{II} + \mathrm{III} + \mathrm{IV}.
		\end{align*}
		First we note that as a consequence of~\eqref{eq:spat_reg_X}, \eqref{eq:fully_discrete_spatial_regularity} and Assumption~\ref{assumptions:3}\ref{ass:3:phi},
		\begin{equation*}
		\sup_{\substack{s \in [0,1] \\ h, \Delta t \in (0,1]}}
		\norm{\hat{\phi}_1(s)}{L^2(\Omega,{\dot{H}^0})} < \infty.
		\end{equation*}
		Therefore, by H\"older's inequality, Lemma~\ref{lem:semigroup_error} and Assumption~\ref{assumptions:1}\ref{assumptions:1:x0},
		\begin{align*}
		\mathrm{I} &\le \sup_{\substack{s \in [0,1] \\ h, \Delta t \in (0,1]}} \norm{\hat{\phi}_1(s)}{L^2(\Omega,{\dot{H}^0})} \norm{P^1(\tilde{E}(T)-E(T))\Theta^{-\frac{\delta}{2}}}{\cL(\cH,\dot{H}^0)} \norm{\Theta^{\frac{\delta}{2}} x_0}{\cH} \\
		&\lesssim h^{\min(\delta \frac{\kappa}{\kappa+1},\kappa)} + {\Delta t}^{\min(\delta \frac{\rho}{\rho+1},1)}.
		\end{align*}
		By the same arguments, we obtain
		\begin{align*}
		\mathrm{II} &\lesssim
		\int^T_0 \norm{P^1\tilde{E}(T-r)BF(\dfloor{r}, \tilde{X}_1(\dfloor{r}))-P^1 E(T-r)BF(r, X_1(r)) }{L^2(\Omega,{\dot{H}^0})} \dd r \\
		&= \sum_{j=0}^{N_{\Delta t}-1} \int_{t_j}^{t_{j+1}} \norm{P^1\left(\tilde{E}(T-r)BF(t_{j}, \tilde{X}_1(t_{j})) - E(T-r)BF(r, X_1(r))\right)}{L^p(\Omega,\dot{H}^0)} \dd r,
		\end{align*}
		and we split the integrand as follows:
		\begin{align*}
		&\norm{P^1\left(\tilde{E}(T-r)BF(t_{j}, \tilde{X}_1(t_{j})) - E(T-r)BF(r, X_1(r))\right)}{L^p(\Omega,\dot{H}^0)} \\
		&\quad \le 
		\norm{P^1\tilde{E}(T-r)B\left(F(t_{j}, \tilde{X}_1(t_j))-F(t_j, {X}_1(t_j)) \right)}{L^p(\Omega,\dot{H}^0)}\\
		&\quad\quad+ \norm{P^1\tilde{E}(T-r)B\left(F(t_j, {X}_1(t_j))-F(r, X_1(t_{j})) \right)}{L^p(\Omega,\dot{H}^0)}  \\
		&\quad\quad+ \norm{P^1\tilde{E}(T-r)B\left(F(r, X_1(t_{j}))-F(r, X_1(r)) \right)}{L^p(\Omega,\dot{H}^0)} 
		\\
		&\quad\quad+ 
		\norm{P^1\left(\tilde{E}(T-r)-E(T-r)\right)BF(r, X_1(r))}{L^p(\Omega,\dot{H}^0)} \\
		&\quad= \mathrm{V} + \mathrm{VI} + \mathrm{VII}+ \mathrm{VIII}.	
		\end{align*}
		Next, by~\eqref{eq:fully_discrete_stability}, the mean value theorem and~\eqref{eq:ass2:munu}, since $\nu \le \min(r,1)$, it follows that
		\begin{align*}
		\mathrm{V} &= \norm{P^1\tilde{E}(T-r)P_hB\left(F(t_{j}, \tilde{X}_1(t_j))-F(t_j, {X}_1(t_j)) \right)}{L^p(\Omega,\dot{H}^0)} \\
		&\le \norm{P^1\tilde E(T-r)}{\cL(\cH,\dot{H}^0)} \norm{P_h B \Lambda^{\frac{\mu}{2}}}{\cL(\dot{H}^0,\cH)} \\
		&\quad\quad\quad\times
		\norm{\Lambda^{-\frac{\mu}{2}}\left(F(t_{j}, \tilde{X}_1(t_{j}))-F(t_{j}, {X}_1(t_{j})) \right)}{L^p(\Omega,\dot{H}^0)} \\
		&\lesssim \norm{P_h B \Lambda^{\frac{\mu}{2}}}{\cL(\dot{H}^0,\cH)} \\
		&\quad\quad\quad\times
		\int^1_0 \norm{\Lambda^{-\frac{\mu}{2}} F'\left(t_j,\tau\tilde{X}_1(t_{j})+(1-\tau){X}_1(t_{j})\right)\left(\tilde{X}_1(t_{j})-{X}_1(t_{j})\right)}{L^p(\Omega,\dot{H}^0)} \dd \tau \\
		&\lesssim \norm{P_h B \Lambda^{\frac{\mu}{2}}}{\cL(\dot{H}^0,\cH)} 
		\norm{ \Lambda^{-\frac{\nu}{2}} \left( \tilde{X}_1(t_{j})-{X}_1(t_{j})\right)}{L^p(\Omega,\dot{H}^0)},
		\end{align*}
		where we have also used Theorems~\ref{thm:X_spat_reg} and~\ref{thm:strong_convergence}. 
		If $\mu \le 1$, then, by~\eqref{eq:B_bound_2}, 
		\begin{equation*}
		\norm{P_h B \Lambda^{\frac{\mu}{2}}}{\cL(\dot{H}^0,\cH)} \le \norm{B \Lambda^\frac{1}{2}}{\cL(\dot{H}^0,\cH)} \norm{\Lambda^{\frac{\mu-1}{2}}}{\cL(\dot{H}^0)} \lesssim 1,
		\end{equation*}
		while if $1<\mu \le 2$, by~\eqref{eq:Theta_h_inv_ineq} and~\eqref{eq:B_bound_2}, it follows that
		\begin{align*}
		\norm{P_h B \Lambda^{\frac{\mu}{2}}}{\cL(\dot{H}^0,\cH)} &= \norm{P_h \Theta^{\frac{\mu-1}{2}} B \Lambda^\frac{1}{2}}{\cL(\dot{H}^0,\cH)} \le \norm{P_h \Theta^{\frac{\mu-1}{2}}}{\cL(\cH)} \norm{B \Lambda^\frac{1}{2}}{\cL(\dot{H}^0,\cH)} \lesssim h^{1-\mu} .
		\end{align*} Next, we similarly derive by~\eqref{eq:fully_discrete_stability} and~\eqref{eq:B_bound_2}, that
		\begin{align*}
		\mathrm{VI} &\lesssim \norm{\Lambda^{-\frac{1 }{2}}\left(F(t_{j}, {X}_1(t_j))-F(r, {X}_1(t_j)) \right)}{L^p(\Omega,\dot{H}^0)} \lesssim \left(1 + \norm{X_1(t_j)}{L^p(\Omega,\dot{H}^0)}\right) |t_j-r|^\eta \\ &\lesssim |t_j-r|^\eta,
		\end{align*}
		where we applied Theorem~\ref{thm:X_spat_reg} and Assumption~\ref{assumptions:1}\ref{assumptions:1:F} in the last inequality. Term $\mathrm{VII}$ is treated like $\mathrm{V}$, and thus, if $\mu \le 1$,
		\begin{align*}
		\mathrm{VII} &\lesssim 
		\left(1 + \norm{\Lambda^{\frac{\nu}{2}}X_1(t_{j})}{L^p(\Omega,\dot{H}^0)} + \norm{\Lambda^{\frac{\nu}{2}}X_1(r)}{L^p(\Omega,\dot{H}^0)}\right) \norm{ \Lambda^{-\frac{\nu}{2}} \left( X_1(t_{j})-{X}_1(r)\right)}{L^p(\Omega,\dot{H}^0)} \\
		& \lesssim  |t_j - r|^{\min(2\nu,1)},
		\end{align*}
		using Theorem~\ref{thm:X_spat_reg} with $\nu \le r$ and Theorem~\ref{thm:mild_holder} in the last step. On the other hand, if $1<\mu \le 2$,
		\begin{equation*}
		\mathrm{VII} \lesssim h^{1-\mu} |t_j - r|^{\min(2\nu,1)}.
		\end{equation*}
		Term $\mathrm{VIII}$ is handled by~\eqref{eq:B_bound_2}, Assumption~\ref{assumptions:1}\ref{assumptions:1:F}, Lemma~\ref{lem:semigroup_error} and Theorem~\ref{thm:X_spat_reg} yielding the estimate, since $\theta \le r$,
		\begin{align*}
		\mathrm{VIII} &\le
		\norm{P^1\left(\tilde{E}(T-r)-E(T-r)\right)B \Lambda^{-\frac{\theta}{2}}}{\cL(\dot{H}^0)} \norm{ \Lambda^{\frac{\theta}{2}} F(t_{j}, {X}_1(r))}{L^p(\Omega,\dot{H}^0)} \\
		&\le \norm{P^1\left(\tilde{E}(T-r)-E(T-r)\right)\Theta^{-\frac{1+\theta}{2}}}{\cL(\cH,\dot{H}^0)} \left(1 + \norm{ \Lambda^{\frac{\theta}{2}}  {X}_1(r)}{L^p(\Omega,\dot{H}^0)} \right) \\
		&\lesssim h^{(1+\theta)\frac{\kappa}{\kappa+1}} + {\Delta t}^{\min((1+\theta)\frac{\rho}{\rho+1},1)}.
		\end{align*}
		In summary, if $\mu \le 1$, we get for $\mathrm{II}$, using also Theorem~\ref{thm:strong_convergence}, 
		\begin{align*}
		\mathrm{II} &\lesssim h^{(1+\theta)\frac{\kappa}{\kappa+1}} + {\Delta t}^{\min( (1+\theta)\frac{\rho}{\rho+1},2\nu,\eta,1)}  +  \Delta t \sum_{j=0}^{N_{\Delta t}-1} \norm{\tilde X_1(t_j)-X_1(t_j)}{L^2(\Omega,\dot{H}^{-\nu})} \\
		&\lesssim h^{r'\frac{\kappa}{\kappa+1}} + {\Delta t}^{\min( r'\frac{\rho}{\rho+1},\eta,1)}.
		\end{align*}
		For $1 < \mu \le 2$, we instead obtain
		\begin{align*}
		\mathrm{II} &\lesssim h^{r'\frac{\kappa}{\kappa+1}+1-\mu} + h^{1-\mu} {\Delta t}^{\min( r'\frac{\rho}{\rho+1},\eta,1)}.
		\end{align*}
		We now continue with term $\mathrm{III}$, which by~\eqref{eq:HSvTrace},~\eqref{eq:trace_cyclic},~\eqref{eq:schatten_bound_1} and~\eqref{eq:schatten_bound_2} satisfies
		\begin{align*}
		\mathrm{III} &= \int^1_0 \int^T_0 \E \Big[  \Big| \trace\Big((1-s) P^1\left(\tilde{E}(T-r)-E(T-r)\right)B Q \left(\hat{\phi}_2(s) D_r X_1(T)\right)^* \Big) \Big| \Big] \dd r \dd s \\
		&= \int^1_0 \int^T_0 (1-s) \E \Big[  \Big| \trace\Big(\hat{\phi}_2(s)^* P^1\left(\tilde{E}(T-r)-E(T-r)\right) B Q \left(D_r X_1(T)\right)^* \Big) \Big| \Big] \dd r \dd s \\
		&\le \int^1_0 \int^T_0
		\E \Big[  \lrnorm{\hat{\phi}_2(s)^* P^1\left(\tilde{E}(T-r)-E(T-r)\right) B Q \left(D_r X_1(T)\right)^*}{\trace} \Big] \dd r \dd s \\
		&= \int^1_0 \int^T_0
		\E \Big[  \Bignorm{\hat{\phi}_2(s)^* P^1\left(\tilde{E}(T-r)-E(T-r)\right)\Theta^{-\beta}  B \Lambda^{\frac{1}{2}} \Lambda^{\beta-\frac{1}{2}} Q \Lambda^{-\frac{1}{2}} \\
			&\hspace{25mm}\times\left(D_r X_1(T)\Lambda^{\frac{1}{2}}\right)^*}{\trace} \Big] \dd r \dd s.
		\end{align*}
		Tonelli's theorem, H\"older's inequality and Jensen's inequality imply
		\begin{align*} 
		\mathrm{III} &\le \int^1_0 \int^T_0 \E \Big[ \norm{\hat{\phi}_2(s)^*}{\cL(\dot{H}^0)} \norm{D_r X_1(T)\Lambda^{\frac{1}{2}}}{\cL(\dot{H}^0)} \Big] \\
		&\hspace{25mm}\times\lrnorm{P^1\left(\tilde{E}(T-r)-E(T-r)\right)\Theta^{-\beta}  B \Lambda^{\frac{1}{2}} \Lambda^{\beta-\frac{1}{2}} Q \Lambda^{-\frac{1}{2}}}{\trace} \dd r \dd s\\
		&\le \sup_{\tau \in [0,T]} \lrnorm{P^1\left(\tilde{E}(T-\tau)-E(T-\tau)\right)\Theta^{-\beta}  B \Lambda^{\frac{1}{2}} \Lambda^{\beta-\frac{1}{2}} Q \Lambda^{-\frac{1}{2}}}{\trace} \\
		&\hspace{25mm}\times \E\left[ \int^1_0 \norm{\hat{\phi}_2(s)^*}{\cL(\dot{H}^0)} \dd s \int^T_0 \norm{D_r X_1(T)\Lambda^{\frac{1}{2}}}{\cL(\dot{H}^0)} \dd r \right]\\
		&\lesssim \sup_{\tau \in [0,T]} \lrnorm{P^1\left(\tilde{E}(T-\tau)-E(T-\tau)\right)\Theta^{-\beta}  B \Lambda^{\frac{1}{2}} \Lambda^{\beta-\frac{1}{2}} Q \Lambda^{-\frac{1}{2}}}{\trace} \\
		&\hspace{25mm}\times \sup_{h, \Delta t \in (0,1]}  \norm{\hat{\phi}_2}{L^2([0,1] \times \Omega,\cL(\dot{H}^0)} \norm{D X_1(T)\Lambda^{\frac{1}{2}}}{L^2([0,T]\times\Omega,\cL(\dot{H}^0))}  \\
		& \lesssim \sup_{\tau \in [0,T]} \lrnorm{P^1\left(\tilde{E}(T-\tau)-E(T-\tau)\right)\Theta^{-\beta}  B \Lambda^{\frac{1}{2}} \Lambda^{\beta-\frac{1}{2}} Q \Lambda^{-\frac{1}{2}}}{\trace},
		\end{align*}
		where the final inequality follows from Lemma~\ref{lem:Malliavin_2} and  
		the fact that by~\eqref{eq:spat_reg_X}, \eqref{eq:fully_discrete_spatial_regularity} and Assumption~\ref{assumptions:3}\ref{ass:3:phi},
		\begin{equation*}
		\sup_{h, \Delta t \in (0,1]} \norm{\hat{\phi}_2}{L^2([0,1] \times \Omega,{\cL(\dot{H}^0)})} < \infty.
		\end{equation*}
		For any $s \in [0,T]$, by~\eqref{eq:schatten_bound_1} and~\eqref{eq:B_bound_2},
		\begin{align*}
		&\lrnorm{P^1\left(\tilde{E}(T-\tau)-E(T-\tau)\right)\Theta^{-\beta}  B \Lambda^{\frac{1}{2}} \Lambda^{\beta-\frac{1}{2}} Q \Lambda^{-\frac{1}{2}}}{\trace} \\
		&\quad\le \norm{P^1\left(\tilde{E}(T-\tau)-E(T-\tau)\right)\Theta^{-\beta}  }{\cL(\cH,\dot{H}^0)} \lrnorm{\Lambda^{\beta-\frac{1}{2}} Q \Lambda^{-\frac{1}{2}}}{\trace},
		\end{align*}
		and thus, by Lemma~\ref{lem:semigroup_error} and Assumption~\ref{assumptions:3}\ref{ass:3:origQ},
		\begin{equation*}
		\mathrm{III} \lesssim h^{\min(2 \beta \frac{\kappa}{\kappa+1},\kappa)} + {\Delta t}^{\min(2 \beta \frac{\rho}{\rho+1},1)}.
		\end{equation*} 
		In the exact same way, one deduces that
		\begin{equation*}
		\mathrm{IV} \lesssim h^{\min(2 \beta \frac{\kappa}{\kappa+1},\kappa)} + {\Delta t}^{\min(2 \beta \frac{\rho}{\rho+1},1)},
		\end{equation*} 
		and since $2\beta \ge \max(2\nu,\beta)$ this finishes the proof in the case when Assumption~\ref{assumptions:3}\ref{ass:3:origQ} is used. 
		
		If Assumption~\ref{assumptions:3}\ref{ass:3:altQ} holds in place of Assumption~\ref{assumptions:3}\ref{ass:3:origQ}, terms $\mathrm{I}-\mathrm{II}$ are analyzed in the same way. For $\mathrm{III}$, we proceed similarly as before, except for that we use the commutativity condition on $\phi''$ along with~\eqref{eq:trace_cyclic} and Lemma~\ref{lem:Malliavin_2} to deduce that
		\begin{align*}
		\mathrm{III} 
		&\le \int^1_0 \int^T_0
		\E \Big[  \lrnorm{\left(\Lambda^{\frac{1}{2}} D_r X_1(T)\right)^* \hat{\phi}_2(s)^* \Lambda^{-\frac{1}{2}} P^1\left(\tilde{E}(T-r)-E(T-r)\right) B Q}{\trace} \Big] \dd r \dd s \\
		&\lesssim \sup_{r \in [0,T]} \lrnorm{\Lambda^{-\frac{1}{2}} P^1\left(\tilde{E}(T-r)-E(T-r)\right) B Q}{\trace} \\
		&\le \trace(Q) \sup_{r \in [0,T]} \lrnorm{\Lambda^{-\frac{1}{2}} P^1\left(\tilde{E}(T-r)-E(T-r)\right) B}{\cL(\dot{H}^0)}.
		\end{align*}
		Therefore, by Lemma~\ref{lem:semigroup_error},
		\begin{equation*}
		\mathrm{III} \lesssim h^{\min(2 \frac{\kappa}{\kappa+1})} + {\Delta t}.
		\end{equation*} 
		Finally, term $\mathrm{IV}$ is treated the same way as above, which finishes the proof.
	\end{proof}
	

	\section{Examples and numerical simulation}
	\label{sec:examples_simulation}
	
	In this section we outline a few examples for which our theory yields weak convergence rates that are greater than the available strong convergence rates. Continuing in the setting of the previous section, where $\dot{H}^0 = L^2(\cD)$ for a domain $\cD \in \R^d$, $d= 1,2,3$, we recall that $G(t) = I$ for all $t \in [0,T]$. 
	
	Below, we also only consider time-independent $F$, so that $\eta$ can be chosen arbitrarily large. We take $F$ to be a Nemytskij operator, which for $u \in \dot{H}^0$ are given by $F(u)(x) = f(u(x))$ for a.e.\ $x \in \cD$. Here $f\colon\R \to \R$ is a  differentiable function such that, for a constant $C>0$, $|f(x)| \le C (1+|x|)$, $|f'(x)| \le C$ and $|f'(x) - f'(y)| \le C |x-y|$ for all $x,y \in \R$. In Appendix~\ref{sec:appendix}, we show that with these conditions on $f$, Assumption~\ref{assumptions:1}\ref{assumptions:1:F} is fulfilled for all $\theta \in [0,1/2)$. If it also holds that $f(0) = 0$, then the assumption is also fulfilled for $\theta \in (1/2,1]$. Moreover, we show that the derivative of $F$, given by $(F'(u)v) (x) = f'(u(x))v(x)$ for $v\in\dot{H}^0$ and a.e.\ $x \in \cD$,  fulfills Assumption~\ref{assumptions:2} for all $\nu \in [0,1/2) \cup (1/2,1)$ and $\mu$ such that $\mu \ge \max(\nu,d/2+\epsilon)$ for an arbitrary small number $\epsilon > 0$. 
	
	\subsection{The white noise case}
	\label{sec:white_noise_1d}
	
	Suppose that $Q=I$, so that we are considering space-time white noise. For Assumption~\ref{assumptions:1}\ref{assumptions:1:G} to be fulfilled we then must have $d=1$, and Assumption~\ref{assumptions:3} is fulfilled for all $\beta < 1/2$, see~\cite[Remark 4.6]{KLL13}. Suppose that $\delta = 1$ so that $r = \beta$, and choose $\nu = r$ maximal. For $0<\epsilon \ll 1/2$ we set $\mu = 1/d + \epsilon = 1/2 + \epsilon > \nu$. Theorem~\ref{thm:weak_convergence} therefore yields the weak convergence result 
	\begin{equation*}
	\left|\E\left[\phi(X(T))-\phi(\tilde X(T))\right]\right| \lesssim  h^{2 \beta \frac{\kappa}{\kappa+1}} + {\Delta t}^{2 \beta \frac{\rho}{\rho+1}}.
	\end{equation*} In contrast, Theorem~\ref{thm:strong_convergence} ensures that
	\begin{equation*}
	\norm{\tilde{X_1}(T)-X_1(T)}{L^2(\Omega,\dot{H}^{0})} \lesssim h^{\beta \frac{\kappa}{\kappa+1}} + {\Delta t}^{\beta \frac{\rho}{\rho+1}}.
	\end{equation*}
	We note that since $2 \beta < 1+\theta$, the value of $\theta$ has no influence on the convergence rate in this case.
	
	Below we illustrate this case with $\cD=(0,1), T=1$. We choose $u_0(x) = x\chi_{[(0,1/2)}(x) + (1-x)\chi_{[1/2,1)}(x)$ and $v_0(x) = \chi_{[(0,1/2)}(x)$, $x \in \cD$. With these choices $X_0 = [u_0,v_0]^\top \in \cH^1$. Moreover, we set $f = \cos(\cdot)$ and use piecewise linear finite elements (i.e., $\kappa=2$) and the Crank--Nicolson method (i.e., $\rho = 2$) in our approximation. See Figure~\ref{fig:white_plot} for a sample of $\tilde X$ with these parameters. 
	
	\begin{figure}[ht]
		\centering
		\includegraphics[width = 1.0\textwidth]{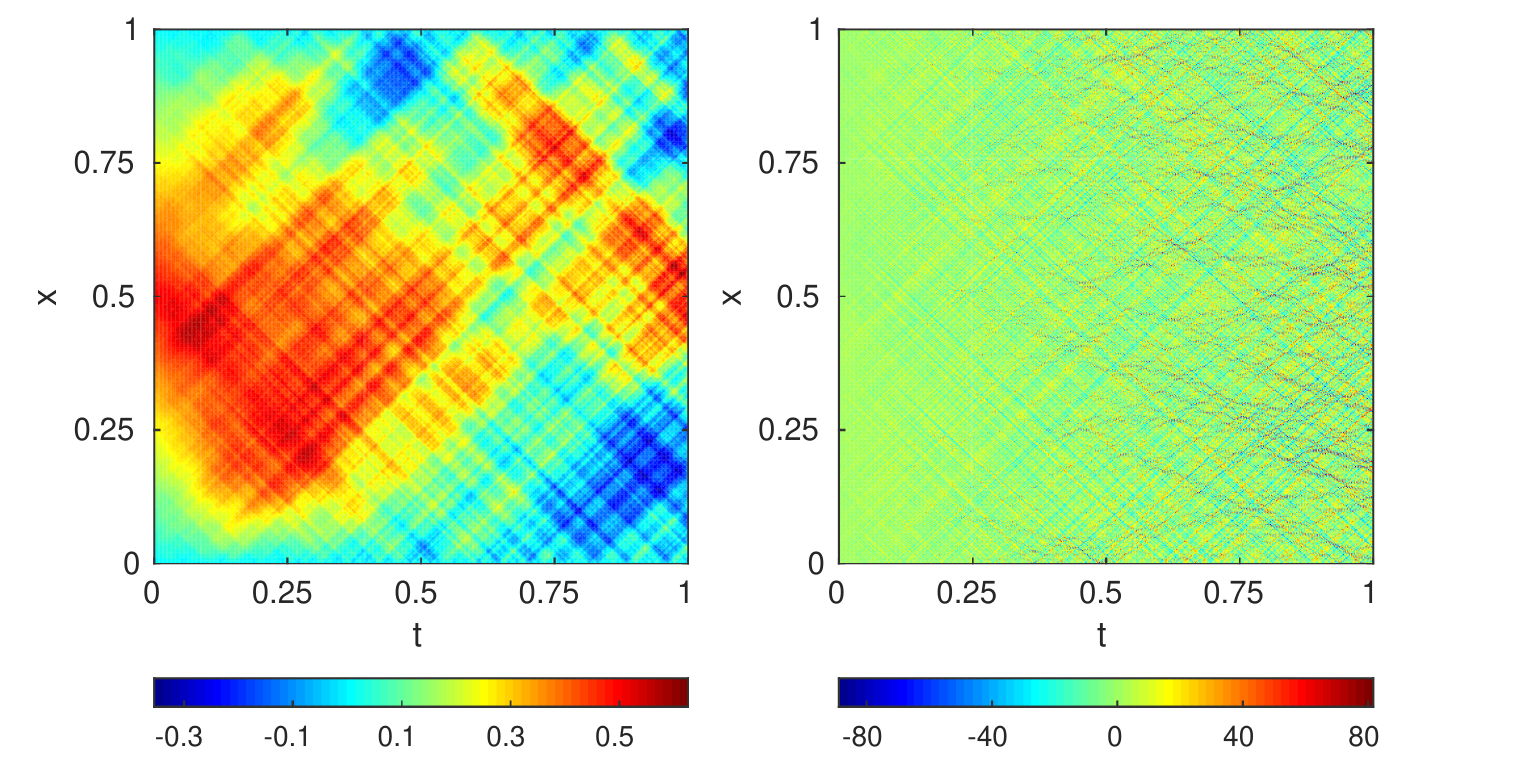}
		\caption{A sample of $\tilde u$ (left) and $\tilde v$ (right) for $\tilde X = [\tilde u, \tilde v]^\top$ computed with the parameters of Section~\ref{sec:white_noise_1d} and $h = \Delta t = 2^{-9}$.}
		\label{fig:white_plot}
	\end{figure}
	
	Choosing $\phi = \|\cdot\|^2$ we approximate our weak error by the Monte Carlo estimate
	\begin{equation*}
	\left|E_N\left[\norm{P^1 X^{N_{\Delta t}}_{h,\Delta t}}{\dot{H}^0}^2-\norm{P^1 X^{N_{\Delta t'}}_{h',\Delta t'}}{\dot{H}^0}^2\right]\right| = 
	\left| \frac{1}{N} \sum^N_{i=1} \left(\norm{P^1 X^{N_{\Delta t}}_{h,\Delta t}}{\dot{H}^0}^2-\norm{P^1 X^{N_{\Delta t'}}_{h',\Delta t'}}{\dot{H}^0}^2\right)^{(i)} \right|,
	\end{equation*}
	where $N$ is the number of iid samples $\big(\norm{P^1 X^{N_{\Delta t}}_{h,\Delta t}}{\dot{H}^0}^2-\norm{P^1 X^{N_{\Delta t'}}_{h',\Delta t'}}{\dot{H}^0}^2\big)^{(i)}$ of $\norm{P^1 X^{N_{\Delta t}}_{h,\Delta t}}{\dot{H}^0}^2-\norm{P^1 X^{N_{\Delta t'}}_{h',\Delta t'}}{\dot{H}^0}^2$. The strong error is approximated by
	\begin{equation*}
	\left(E_N\left[\norm{P^1 X^{N_{\Delta t}}_{h,\Delta t}-P^1 X^{N_{\Delta t'}}_{h',\Delta t'}}{\dot{H}^0}^2\right]\right)^{\frac{1}{2}}.
	\end{equation*}
	A reference solution $X^{N_{\Delta t'}}_{h',\Delta t'}$, $h', \Delta t' \in (0,1]$, replaces the analytical solution since this is not available. We set $\Delta t = h$ and compute  errors for $h = 2^{-1}, 2^{-2}, \ldots, 2^{-6}$. We use a reference solution with $\Delta t' = h' = 2^{-8}$ and use $N = 2000$ samples in our Monte Carlo simulation. As one can see from  Figure~\ref{fig:white_errors}, the behaviour of the strong errors are consistent with our theoretical results while the weak errors appear to decay faster than expected. This is in line with~\cite{W15} where numerical convergence rates of 1 were reported for a Crank--Nicolson discretization of the stochastic wave equation driven by white noise. 
	
	\begin{figure}[ht]
		\centering
		\includegraphics[width = .75\textwidth]{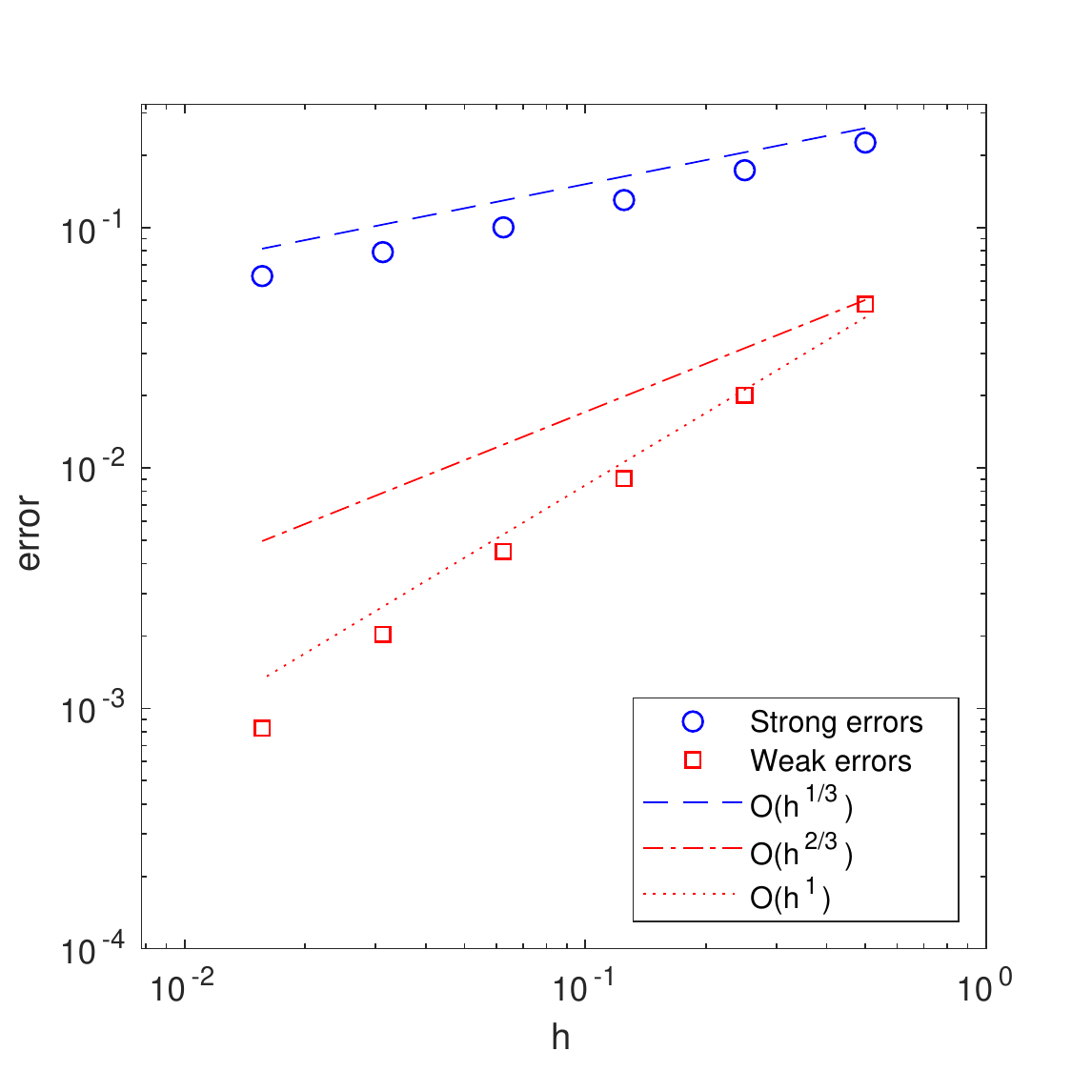}
		\caption{Monte Carlo estimates of strong and weak errors in the setting of Section~\ref{sec:white_noise_1d}: The white noise case.}
		\label{fig:white_errors}
	\end{figure}
	
	\subsection{The trace-class noise case}
	\label{sec:trace_noise_1d}
	
	If we assume that $Q$ is of trace-class, Assumption~\ref{assumptions:1}\ref{assumptions:1:G} holds for all $\beta \le 1$ but in general not for $\beta>1$. If $\Lambda Q = Q \Lambda$, or if $\Lambda Q \neq Q \Lambda$ with $\phi$ as in Section~\ref{sec:white_noise_1d}, then Assumption~\ref{assumptions:3}\ref{ass:3:origQ} or~\ref{ass:3:altQ} is fulfilled for $\beta \le 1$, respectively. Let us take $\beta = 1$ and suppose that $\min(1+\theta,\delta) \ge 2$ (letting $f(0) = 0$ so that $\theta = 1$), which ensures that $r = \beta = 1$. For arbitrary $0<\epsilon\ll1/2$, we choose $\mu = \max(d/2+\epsilon,1)$. In $d = 1$ we choose $\nu = \beta - \epsilon/2 = 1 - \epsilon/2 < \mu$ so that $r' = 2 \nu = 2 - \epsilon$ in Theorem~\ref{thm:weak_convergence}.  Our weak convergence result in that theorem then states that
	\begin{equation*}
	\left|\E\left[\phi(X_1(T))-\phi(\tilde X_1(T))\right]\right| \lesssim  h^{(2 -\epsilon) \frac{\kappa}{\kappa+1}} + {\Delta t}^{\min((2  - \epsilon)\frac{\rho}{\rho+1},1)},
	\end{equation*}
	while Theorem~\ref{thm:strong_convergence} yields the (for sufficiently small $\epsilon$) slower strong convergence rate
	\begin{equation*}
	\norm{\tilde{X_1}(T)-X_1(T)}{L^2(\Omega,\dot{H}^{0})} \lesssim h^{ \frac{\kappa}{\kappa+1}} + {\Delta t}^{ \frac{\rho}{\rho+1}}.
	\end{equation*}
	In $d = 2$, we choose $\beta$  and $\nu$ as before and $\mu = 1 + \epsilon > \nu$. Our strong convergence result remains the same as in $d=1$ while the weak rate becomes 
	\begin{equation*}
	\left|\E\left[\phi(X_1(T))-\phi(\tilde X_1(T))\right]\right| \lesssim h^{(2-\epsilon) \frac{\kappa}{\kappa+1}-\epsilon} + h^{-\epsilon} {\Delta t}.
	\end{equation*}
	Note that in both $d=1$ and $d=2$, the Crank--Nicolson scheme provides no essential benefit, in terms of the weak convergence rate, over the backward Euler scheme in this setting. In either case we have a weak rate that is essentially twice as big as the strong rate. In the case $d = 3$, however, we need to have $\mu > 3/2$, which means that we get a factor of $h^{-\frac{1}{2}}$ in Theorem~\ref{thm:weak_convergence}. Therefore, while Theorem~\ref{thm:weak_convergence} still yields greater spatial convergence rates compared to Theorem~\ref{thm:strong_convergence} for appropriate parameter configurations, the temporal convergence rate will be significantly lower.
	
	In $d=1$, we now compute weak and strong errors numerically in the setting outlined above with the same choices of $\cD$, $T$, $\rho$ and $\kappa$ as in Section~\ref{sec:white_noise_1d}. Let $Q$ be the integral operator defined by 
	\begin{equation*}
	\inpro[\dot{H}^0]{Qu}{v} = \int_{\cD \times \cD} q(x,y) u(x) v(y) \dd x \dd y
	\end{equation*} 
	for all $u,v \in \dot{H}^0$. We choose, for $x,y \in \cD$, the exponential covariance kernel $q(x,y) = q(x-y) = \exp(-25|x-y|)/16$, $f(x) = \sin(x)$, and $u_0 = v_0 = 0$. See Figure~\ref{fig:trace_plot} for a sample of $\tilde X$ with these parameters. 
	
	\begin{figure}[ht]
		\centering
		\includegraphics[width = 1.0\textwidth]{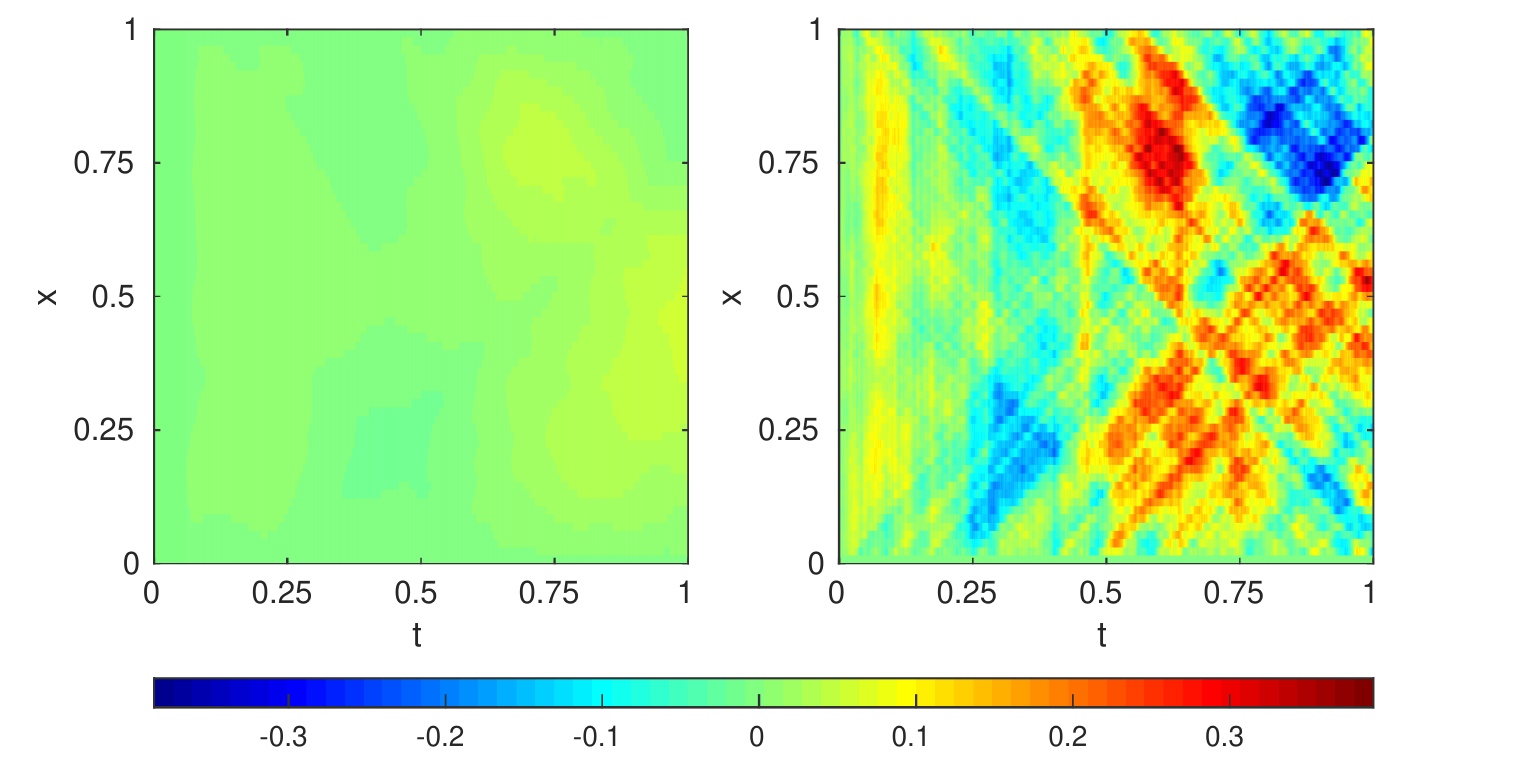}
		\caption{A sample of $\tilde u$ (left) and $\tilde v$ (right) for $\tilde X = [\tilde u, \tilde v]'$ computed with the parameters of Section~\ref{sec:trace_noise_1d} and $h^2 = \Delta t = 2^{-12}$.}
		\label{fig:trace_plot}
	\end{figure}
	
	The temporal step size is set to $\Delta t = h^2$. With this choice, we expect to see a weak and strong convergence rate of approximate order $\Op(h^{4/3})$ and $\Op(h^{2/3})$, respectively. We compute  errors for $h = 2^{-1}, 2^{-2}, \ldots, 2^{-5}$ and use a reference solution with $\Delta t' = {h'}^2 = 2^{-12}$, employing $N = 500$ samples in our Monte Carlo simulations. As one can see from  Figure~\ref{fig:trace_errors}, the decay of the errors is consistent with our theoretical results.
	
	\begin{figure}[ht]
		\centering
		\includegraphics[width = .75\textwidth]{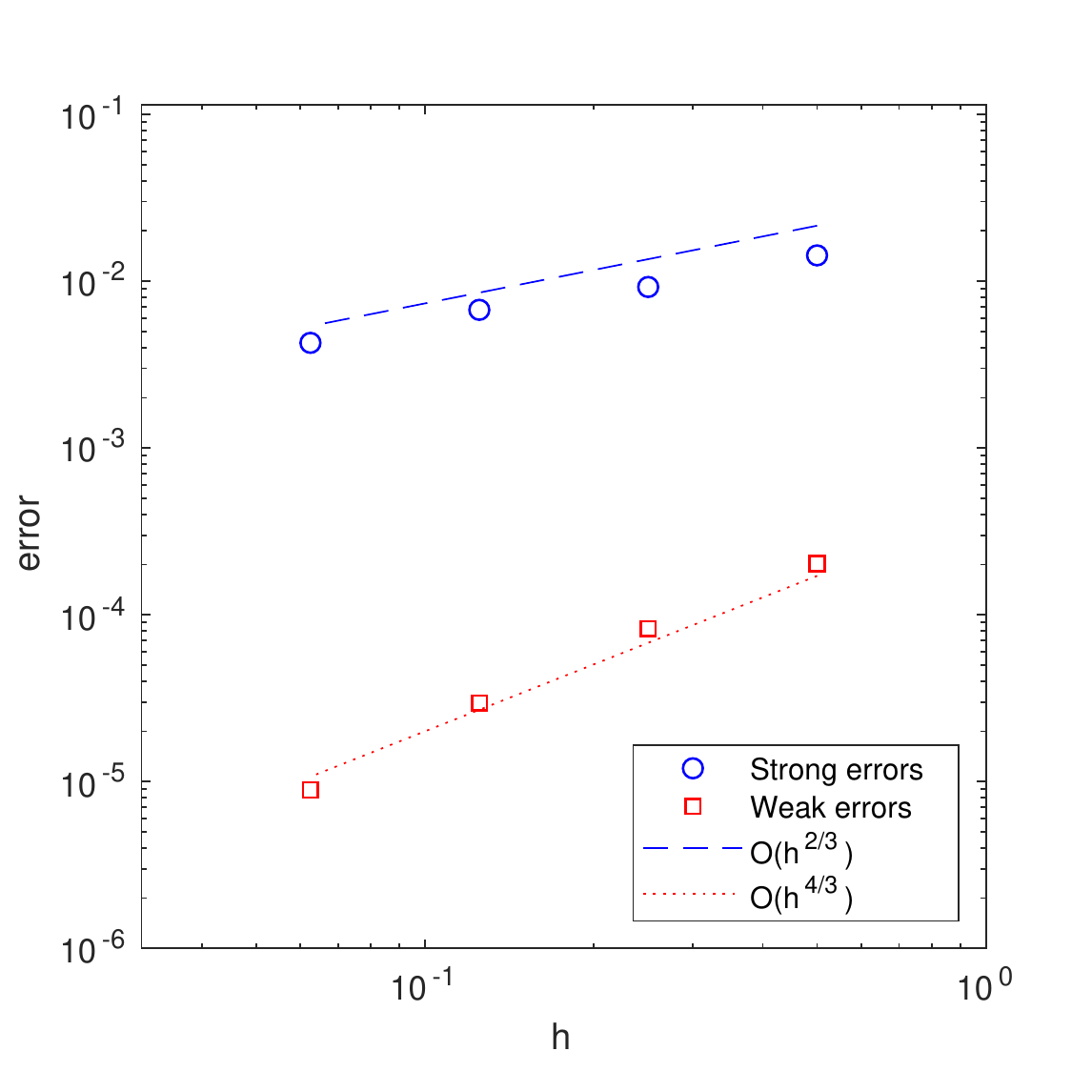}
		\caption{Monte Carlo estimates of strong and weak errors in the setting of Section~\ref{sec:trace_noise_1d}: The trace-class noise case.}
		\label{fig:trace_errors}
	\end{figure}

	\appendix
	\input{appendix}
	
	\bibliographystyle{hplain}
	\bibliography{wave-cov}
	
\end{document}

%% file: appendix.tex
\section{Nemytskij operators}
\label{sec:appendix}

Let  $\dot{H}^0 = L^2(\cD)$ for a convex bounded
 domain $\cD$ in $\R^d$, $d=1,2,3$.
For $m \in \N_0$, we denote by $H^m = H^m(\cD) = W^{m,2}(\cD)$ the classical Sobolev space of order $m$. For $s = m + \sigma$, $m \in \N_0$, $\sigma \in (0,1)$, we employ the same notation for the fractional Sobolev space $H^s$ (see \cite{NPV12}) equipped with the Sobolev--Slobodeckij norm
\begin{equation*}
\norm{u}{H^s} = \left(\norm{u}{H^m}^2 + \sum_{|\alpha| = m} \int_{\cD \times \cD} \frac{|D^\alpha u(x) - D^\alpha u(y)|^2}{|x-y|^{d+2\sigma}}\dd x \dd y\right)^{1/2},
\end{equation*}
where $u \in H^s$. The spaces $(\dot{H}^{s})_{s \in [0,2]}$ are related to $(H^{s})_{s \in [0,2]}$ by (see, e.g., \cite[Theorem~4.5]{Y08})
\begin{equation}
\label{eq:sobolev_dot_equivalence}
	\dot{H}^s = 
	\begin{cases}
		H^s \text{ if } s \in [0,1/2), \\
		\left\{u \in H^s : u = 0 \text{ a.e. on } \partial \cD \right\} \text{ if } s \in (1/2,3/2) \cup (3/2,2],
	\end{cases}	
\end{equation}
with norm equivalence.

The aim of this appendix is to show some results on \emph{Nemytskij operators}, i.e., operators $F$ that are for $u \in \dot{H}^0$ given by $F(u)(x) = f(u(x))$ for almost every $x \in \cD$, where $f:\R \to \R$ is a measurable function. We assume $f$ to be Lipschitz continuous, i.e., that there exists a constant $C > 0$ such that
\begin{equation}
\label{eq:appendix_f_bound_lip}
|f(x)-f(y)| \le C |x-y|,
\end{equation}
for all $x,y \in \R$, implying the existence of a constant $C > 0$ such that
\begin{equation}
\label{eq:appendix_f_bound_lin_growth}
|f(x)| \le C (1+|x|),
\end{equation}
for all $x \in \R$.

If $f$ is also a once continuously differentiable function with a bounded first derivate, i.e., if $f': \R \to \R$ is continuous and there exists a constant $C>0$ such that $|f'(x)| \le C$ for all $x\in \R$, then $F \in \cG^1_\mathrm{b}(\dot{H}^0)$ (see e.g., \cite[Theorem~2.7, Chapter~1]{AP95}) and the derivative of $F$ at $u$ is given by 
\begin{equation*}
	(F'(u)v)(x) = f'(u(x))v(x)
\end{equation*}
or all $v \in H$ and almost every $x \in \cD$. 

We first show that $F$ is Lipschitz continuous on $\dot{H}^0$ and that it fulfills a linear growth condition.

\begin{proposition}
	\label{prop:appendix_1}
	Let $f: \R \to \R$ be a Lipschitz continuous function and let $F$ be the corresponding Nemytskij operator. Then, there exists a constant $C>0$ such that
	\begin{equation}
	\label{eq:appendix_lipschitz}
		\norm{F(u)-F(v)}{\dot{H}^0} \le C \norm{u-v}{\dot{H}^0}
	\end{equation}
	for all $u,v \in \dot{H}^0$ and 
	\begin{equation}
	\label{eq:appendix_bound}
	\norm{F(u)}{\dot{H}^\theta} \le C\left( 1 + \norm{u}{\dot{H}^\theta}\right)
	\end{equation}
	for all $u \in \dot{H}^\theta$, $\theta \in [0,1/2)$. If also $f(0) = 0$, then~\eqref{eq:appendix_bound} holds for $\theta \in (1/2,1)$. If, in addition, $f$ is continuously differentiable with a bounded derivative $f'$, then~\eqref{eq:appendix_bound} holds for $\theta=1$.
\end{proposition}
\begin{proof}
	The inequality~\eqref{eq:appendix_lipschitz} is a direct consequence of~\eqref{eq:appendix_f_bound_lip} via
	\begin{equation*}
		\norm{F(u)-F(v)}{\dot{H}^0}^2 = \int_{\cD} |f(u(x))-f(v(x))|^2 \dd x \le C^2 \int_\cD |u(x)-v(x)|^2 \dd x = C^2 \norm{u-v}{\dot{H}^0}^2.
	\end{equation*}
	For~\eqref{eq:appendix_bound} with $\theta \in [0,1/2)$ we also make use of~\eqref{eq:sobolev_dot_equivalence} and~\eqref{eq:appendix_f_bound_lin_growth} to see that
	\begin{align*}
		\norm{F(u)}{\dot{H}^\theta}^2 &\lesssim \norm{F(u)}{H^\theta}^2 = \norm{F(u)}{\dot{H}^0}^2 + \int_{\cD \times \cD} \frac{|f(u(x))-f(u(y))|^2}{|x-y|^{d + 2 \theta}} \dd x \dd y \\
		&\lesssim 1 + \norm{u}{\dot{H}^0}^2 + \int_{\cD \times \cD} \frac{|u(x)-u(y)|^2}{|x-y|^{d + 2 \theta}} \dd x \dd y = 1 + \norm{u}{H^\theta}^2 \lesssim 1 + \norm{u}{\dot{H}^\theta}^2.
	\end{align*} 
	The same argument is used for \eqref{eq:appendix_bound} when $\theta \in (1/2,1)$, noting that the condition $f(0)=0$ means that $F(u)$ will inherit the boundary condition of $u \in \dot{H}^\theta$. For $\theta = 1$ we simply note that, due to the definition of $\norm{\cdot}{\dot{H}^1}$, the chain rule for weak derivatives and the assumption that $|f'(x)| \le C$ for all $x \in \R$, 
	\begin{equation*}
		\norm{F(u)}{\dot{H}^1}^2 = \int_{\cD} |\nabla f(u(x))|^2 \dd x = \int_{\cD} |f'(u(x))|^2 |\nabla u(x)|^2 \dd x \le C^2 \norm{u}{\dot{H}^1}^2
	\end{equation*}
	which completes the proof of the proposition.
\end{proof}

Next, we show that $F'$ fulfills a negative norm bound if $f'$ is Lipschitz continuous.

\begin{proposition}
	Let $f$ be a continuously differentiable function of at most linear growth with bounded and Lipschitz continuous derivative $f'$ and let $F$ be the corresponding Nemytskij operator. Then, there exists a constant $C>0$ such that
	\begin{equation}
	\label{eq:appendix_derivative_bound}
	\norm{F'(u)v }{\dot{H}^0} \le C \norm{v}{\dot{H}^0},
	\end{equation}
	for all $u,v \in \dot{H}^0$, and 
	\begin{equation}
	\label{eq:appendix_negative_derivative_bound}
	\norm{F'(u)v}{\dot{H}^{-\mu}} \le C \left(1+\norm{u}{\dot{H}^\nu}\right) \norm{v}{\dot{H}^{-\nu}},
	\end{equation}
	for all $u \in \dot{H}^\nu$ and $v \in \dot{H}^{-\nu}$ where $\epsilon > 0$, $\nu \in [0,1/2) \cup (1/2,1)$ and $ \mu \ge \max(d/2+\epsilon,\nu)$. If, in addition, $f'$ is differentiable with a bounded derivative $f''$, then~\eqref{eq:appendix_negative_derivative_bound} holds for $\nu =1$.
\end{proposition}

\begin{proof}
	The first estimate \eqref{eq:appendix_derivative_bound} is a direct consequence of the assumption that there exists a constant $C>0$ such that $|f'(x)|\le C$ for all $x \in \R$ via the estimate 
	\begin{equation*}
		\norm{F'(u)v }{\dot{H}^0}^2
		= \int_{\cD} |f'(u(x)) v(x)|^2 \dd x
		\le C^2 \int_{\cD} |v(x)|^2 \dd x
		= C^2 \norm{v}{\dot{H}^0}
	\end{equation*}
	for all $u, v \in \dot{H}^0$. This also shows~\eqref{eq:appendix_negative_derivative_bound} for $\nu = 0$. For $\nu > 0$, we mimic the approach of~\cite[Lemma~4.4]{WGT14}. Let $u \in \dot{H}^\nu$ and $v \in \dot{H}^\mu$. Then $F'(u)v[x] = f'(u(x))v(x) = 0$ for a.e. $x \in \partial \cD$ since $v(x) = 0$ for a.e. $x \in \partial \cD$, as a consequence of $\mu > 1/2$ and~\eqref{eq:sobolev_dot_equivalence}. We may therefore use~\eqref{eq:sobolev_dot_equivalence} to obtain, if $\nu \in (0,1/2) \cup (1/2,1)$, that
	\begin{equation*}
		\norm{F'(u)v}{\dot{H}^\nu}^2 \lesssim \norm{F'(u)v}{H^\nu}^2 = \norm{F'(u)v}{\dot{H}^0}^2 + \int_{\cD \times \cD} \frac{| f'(u(x))v(x) -  f'(u(y))v(y)|^2}{|x-y|^{d+2\nu}}\dd x \dd y.
	\end{equation*}
	Using \eqref{eq:appendix_derivative_bound}, the inequality $(a+b)^2\le 2(a^2 + b^2)$, $a,b\in\R$, the Lipschitz assumption on $f'$ and the fact that since $\mu > d/2$, by the Sobolev embedding theorem $\dot{H}^\mu \subset L^\infty(\cD)$ continuously, we find that 
	\begin{align*}
	\norm{F'(u)v}{\dot{H}^\nu}^2 &\lesssim \norm{v}{\dot{H}^0}^2 + \int_{\cD \times \cD} \frac{| f'(u(x))(v(x) -  v(y))|^2}{|x-y|^{d+2\nu}}\dd x \dd y \\ 
	&\quad+ \int_{\cD \times \cD} \frac{| (f'(u(x)) -  f'(u(y)))v(y)|^2}{|x-y|^{d+2\nu}}\dd x \dd y \\
	&\lesssim \norm{v}{\dot{H}^0}^2 + \int_{\cD \times \cD} \frac{|(v(x) -  v(y))|^2}{|x-y|^{d+2\nu}}\dd x \dd y \\ 
	&\quad+ \norm{v}{L^\infty(\cD)}^2 \int_{\cD \times \cD} \frac{| u(x) -  u(y)|^2}{|x-y|^{d+2\nu}}\dd x \dd y \\ 
	&\lesssim \norm{v}{\dot{H}^\nu}^2 + \norm{v}{L^\infty(\cD)}^2 \norm{u}{\dot{H}^\nu}^2 \lesssim \norm{v}{\dot{H}^\mu}^2(1 + \norm{u}{\dot{H}^\nu}^2).
	\end{align*}
	If $\nu = 1$ and $f'$ is differentiable with $f''$ bounded, we directly use the definition of $\dot{H}^1$ to see that, by the same arguments as above,
	\begin{align*}
	\norm{F'(u)v}{\dot{H}^\nu}^2 &= \int_\cD |\nabla (f'(u(x))v(x))|^2 \dd x \\
	&\lesssim \int_\cD |f'(u(x))\nabla v(x)|^2 \dd x + \int_\cD |f''(u(x))\nabla u(x) v(x)|^2 \dd x \\
	&\lesssim \norm{v}{\dot{H}^0}^2 + \norm{v}{L^\infty(\cD)}^2 \norm{u}{\dot{H}^1}^2 \lesssim \norm{v}{\dot{H}^\mu}^2(1 + \norm{u}{\dot{H}^\nu}^2).
	\end{align*} 
	In summary,
	\begin{equation*}
		\norm{\Lambda^{\frac{\nu}{2}}F'(u) \Lambda^{-\frac{\mu}{2}}}{\cL(\dot{H}^{0})}^2 =
		\norm{F'(u)}{\cL(\dot{H}^{\mu},\dot{H}^\nu)}^2 \lesssim 1 + \norm{u}{\dot{H}^\nu}^2
	\end{equation*}
	and thus, using that $F'(u)$ is symmetric on $\dot{H}^0$, we have for $v \in \dot{H}^0$ that 
	\begin{align*}
		\norm{\Lambda^{-\frac{\mu}{2}} F'(u) v}{\dot{H}^0} &= \sup_{\substack{w \in \dot{H}^0 \\ \norm{w}{\dot{H}^0}=1}} \left| \inpro[\dot{H}^0]{\Lambda^{-\frac{\mu}{2}} F'(u) v}{w}\right| = \sup_{\substack{w \in \dot{H}^0 \\ \norm{w}{\dot{H}^0}=1}} \left| \inpro[\dot{H}^0]{\Lambda^{-\frac{\nu}{2}}v}{\Lambda^{\frac{\nu}{2}} F'(u)\Lambda^{-\frac{\mu}{2}} w}\right| \\
		&\le \norm{v}{\dot{H}^{-\nu}} \norm{\Lambda^{\frac{\nu}{2}}F'(u) \Lambda^{-\frac{\mu}{2}}}{\cL(\dot{H}^{0})} 
	\end{align*}
	and since $\dot{H}^0$ is dense in $\dot{H}^{-\nu}$, this implies that
	\begin{equation*}
	\norm{F'(u)}{\cL(\dot{H}^{-\nu},\dot{H}^{-\mu})}^2 \lesssim 1 + \norm{u}{\dot{H}^\nu}^2,
	\end{equation*}
	which is equivalent to \eqref{eq:appendix_negative_derivative_bound}.
\end{proof}

%% file: wave-cov.bbl
\begin{thebibliography}{10}

\bibitem{AP95}
Antonio Ambrosetti and Giovanni Prodi.
\newblock {\em A primer of nonlinear analysis}, volume~34 of {\em Cambridge
  Studies in Advanced Mathematics}.
\newblock Cambridge University Press, Cambridge, 1995.
\newblock Corrected reprint of the 1993 original.

\bibitem{AKL16}
Adam Andersson, Raphael Kruse, and Stig Larsson.
\newblock Duality in refined {S}obolev{-}{M}alliavin spaces and weak
  approximation of {S}{P}{D}{E}.
\newblock {\em Stochastics and Partial Differential Equations Analysis and
  Computations}, 4(1):113--149, 2016.

\bibitem{ACLW16}
Rikard Anton, David Cohen, Stig Larsson, and Xiaojie Wang.
\newblock Full discretization of semilinear stochastic wave equations driven by
  multiplicative noise.
\newblock {\em SIAM Journal on Numerical Analysis}, 54(2):1093--1119, 2016.

\bibitem{BB79}
Garth~A. Baker and James~H. Bramble.
\newblock Semidiscrete and single step fully discrete approximations for second
  order hyperbolic equations.
\newblock {\em RAIRO Anal. Num\'{e}r.}, 13(2):75--100, 1979.

\bibitem{B05}
Dirk Bl\"omker.
\newblock Nonhomogeneous noise and {Q}-{W}iener processes on bounded domains.
\newblock {\em Stochastic Analysis and Applications}, 23(2):255--273, 2005.

\bibitem{BKM18}
David Bolin, Kristin Kirchner, and Mih\'aly Kov\'acs.
\newblock {Numerical solution of fractional elliptic stochastic PDEs with
  spatial white noise}.
\newblock {\em IMA Journal of Numerical Analysis}, 12 2018.

\bibitem{CY07}
Yanzhao Cao and Li~Yin.
\newblock Spectral {G}alerkin method for stochastic wave equations driven by
  space-time white noise.
\newblock {\em Commun. Pure Appl. Math.}, 6(3):607--617, 2007.

\bibitem{CLS13}
David Cohen, Stig Larsson, and Magdalena Sigg.
\newblock A trigonometric method for the linear stochastic wave equation.
\newblock {\em SIAM Journal on Numerical Analysis}, 51(1):204--222, 2013.

\bibitem{CQ15}
David Cohen and Llu{\'i}s Quer-Sardanyons.
\newblock {A fully discrete approximation of the one-dimensional stochastic
  wave equation}.
\newblock {\em IMA Journal of Numerical Analysis}, 36(1):400--420, 03 2015.

\bibitem{CJL19}
Sonja {Cox}, Arnulf {Jentzen}, and Felix {Lindner}.
\newblock {Weak convergence rates for temporal numerical approximations of
  stochastic wave equations with multiplicative noise}.
\newblock {\em arXiv e-prints}, page arXiv:1901.05535, Jan 2019, 1901.05535.

\bibitem{DPZ92}
Giuseppe Da~Prato and Jerzy Zabczyk.
\newblock {\em Stochastic {E}quations in {I}nfinite {D}imensions}, volume~44 of
  {\em {E}ncyclopedia of {M}athematics and its {A}pplications}.
\newblock Cambridge University Press, Cambridge, 1992.

\bibitem{DPZ14}
Giuseppe Da~Prato and Jerzy Zabczyk.
\newblock {\em Stochastic {E}quations in {I}nfinite {D}imensions}, volume 152
  of {\em {E}ncyclopedia of {M}athematics and its {A}pplications}.
\newblock Cambridge University Press, Cambridge, second edition, 2014.

\bibitem{D09}
Robert~C. Dalang.
\newblock The stochastic wave equation.
\newblock In Davar Khoshnevisan and Firas Rassoul-Agha, editors, {\em A
  Minicourse on Stochastic Partial Differential Equations}, pages 39--71,
  Berlin, Heidelberg, 2009. Springer Berlin Heidelberg.

\bibitem{G75}
Rolf~Dieter Grigorieff.
\newblock {D}iskrete {A}pproximation von {E}igenwertproblemen.
\newblock {\em {N}umerische {M}athematik}, 25(1):79--97, Mar 1975.

\bibitem{H10}
Erika Hausenblas.
\newblock Weak approximation of the stochastic wave equation.
\newblock {\em Journal of Computational and Applied Mathematics}, 235(1):33 --
  58, 2010.

\bibitem{JJW15}
Ladislas {Jacobe de Naurois}, Arnulf {Jentzen}, and Timo {Welti}.
\newblock {Weak convergence rates for spatial spectral {G}alerkin
  approximations of semilinear stochastic wave equations with multiplicative
  noise}.
\newblock arXiv:1508.05168[math.PR], Aug 2015, 1508.05168.

\bibitem{KLL12}
Mih{\'a}ly Kov{\'a}cs, Stig Larsson, and Fredrik Lindgren.
\newblock Weak convergence of finite element approximations of linear
  stochastic evolution equations with additive noise.
\newblock {\em BIT Numerical Mathematics}, 52(1):85--108, Mar 2012.

\bibitem{KLL13}
Mih{\'a}ly Kov{\'a}cs, Stig Larsson, and Fredrik Lindgren.
\newblock Weak convergence of finite element approximations of linear
  stochastic evolution equations with additive noise {II}. {F}ully discrete
  schemes.
\newblock {\em BIT Numerical Mathematics}, 413(2):497, 2013.

\bibitem{KLS15}
Mih{\'a}ly Kov{\'a}cs, Felix Lindner, and Ren{\'e}~L. Schilling.
\newblock Weak convergence of finite element approximations of linear
  stochastic evolution equations with additive {L}\'evy noise.
\newblock {\em SIAM-ASA Journal on Uncertainty Quantification},
  3(1):1159--1199, 2015.

\bibitem{K14}
Raphael Kruse.
\newblock {\em Strong and Weak Approximation of Semilinear Stochastic Evolution
  Equations}, volume 2093 of {\em Lecture Notes in Mathematics}.
\newblock Springer, 2014.

\bibitem{L12}
Fredrik Lindgren.
\newblock {\em On weak and strong convergence of numerical approximations of
  stochastic partial differential equations.}
\newblock PhD thesis, Chalmers University of Technology, 2012.
\newblock Available at https://research.chalmers.se/publication/.

\bibitem{NPV12}
Eleonora~Di Nezza, Giampiero Palatucci, and Enrico Valdinoci.
\newblock Hitchhikerʼs guide to the fractional {S}obolev spaces.
\newblock {\em Bulletin des Sciences Mathématiques}, 136(5):521 -- 573, 2012.

\bibitem{N06}
David Nualart.
\newblock {\em The {M}alliavin calculus and related topics}.
\newblock Probability and its Applications (New York). Springer-Verlag, Berlin,
  2nd edition, 2006.

\bibitem{PR07}
Claudia Pr{\'e}v{\^o}t and Michael R{\"o}ckner.
\newblock {\em A {Concise} {Course} on {Stochastic} {Partial} {Differential}
  {Equations}}, volume 1905 of {\em Lecture Notes in Mathematics}.
\newblock Springer, Berlin, 2007.

\bibitem{QS06}
Llu{\'i}s Quer-Sardanyons and Marta Sanz-Sol{\'e}.
\newblock Space semi-discretisations for a stochastic wave equation.
\newblock {\em Potential Analysis}, 24(4):303--332, Jun 2006.

\bibitem{T06}
Vidar Thom{\'e}e.
\newblock {\em Galerkin Finite Element Methods for Parabolic Problems},
  volume~25 of {\em Springer Series in Computational Mathematics}.
\newblock Springer, 2nd edition, 2006.

\bibitem{W06}
John~B. Walsh.
\newblock On numerical solutions of the stochastic wave equation.
\newblock {\em Illinois J. Math.}, 50(1-4):991--1018, 2006.

\bibitem{W15}
Xiaojie Wang.
\newblock An exponential integrator scheme for time discretization of nonlinear
  stochastic wave equation.
\newblock {\em Journal of Scientific Computing}, 64(1):234--263, Jul 2015.

\bibitem{WGT14}
Xiaojie Wang, Siqing Gan, and Jingtian Tang.
\newblock Higher order strong approximations of semilinear stochastic wave
  equation with additive space-time white noise.
\newblock {\em SIAM Journal on Scientific Computing}, 36(6):A2611--A2632, 2014.

\bibitem{Y08}
Atsushi Yagi.
\newblock ${H}^\infty$-functional calculus and characterization of domains of
  fractional powers.
\newblock In Tsuyoshi Ando, Ra{\'u}l~E. Curto, Il~Bong Jung, and Woo~Young Lee,
  editors, {\em Recent Advances in Operator Theory and Applications}, pages
  217--235, Basel, 2008. Birkh{\"a}user Basel.

\end{thebibliography}
